\documentclass[letterpaper,12pt]{amsart} 
\usepackage{amsmath}
\usepackage{amssymb}
\usepackage{fix-cm}
\usepackage[left=0.7in,top=0.6in,right=0.5in,bottom=0.7in]{geometry}
\usepackage{hyperref}
\usepackage{amsthm}
\usepackage{graphicx}
\DeclareGraphicsExtensions{.eps}

\newcommand{\comments}[1]{}
\newcounter{parcount}[subsubsection]
\newtheorem{theorem}{Theorem}[section]

\newtheorem{lemma}[theorem]{Lemma}

\DeclareGraphicsExtensions{.eps}

\newtheorem{utv*}{Proposition}
\newtheorem{hyp*}{Conjecture}
\newtheorem{defin}{Definition}
\newtheorem{zamech}[theorem]{Remark}
\newtheorem*{th*}{Theorem}
\newtheorem{prop}{Proposition}

\numberwithin{equation}{section}

\newcommand{\norm}[2]{\left\| #1 \right\|_{#2}}
\newcommand{\mdl}[1]{\left| #1 \right|}

\newcommand{\av}[2]{\left\langle #1\right\rangle_{_{\scriptstyle #2}}}
\newcommand{\ave}[1]{\left\langle #1\right\rangle}

\newcommand{\R}{\mathbb{R}}

\newcommand{\Z}{\mathbb{Z}}

\def\nn{\nonumber}

\def\sli{\sum\limits}
\def\ili{\int\limits}

\def\vf{\varphi}

\title{Equivalent definitions of dyadic Muckenhoupt and Reverse H\"older classes in terms of Carleson sequences, weak classes, and comparability of dyadic $L\log L$ and $A_\infty$ constants.}
\author{O. Beznosova}
\address{Department of Mathematics, Baylor University, One Bear Place \#97328, Waco, TX 76798-7328, USA.}
\author{A. Reznikov}
\address{Department of Mathematics, Michigan State University, East
Lansing, MI 48824, USA}
\address{St.-Petersburg Department of the Steklov Mathematical Institute, Fontanka, 27, 191023, Saint Petersburg, Russia.}
\begin{document}

\maketitle

\begin{abstract}

{

In the dyadic case the union of the Reverse H\"{o}lder classes, $\bigcup_{p>1} RH_p^d$ is strictly larger than the union of the Muckenhoupt classes $\bigcup_{p>1} A_p^d = A_\infty^d$. We introduce the
$RH_1^d$ condition as a limiting case of the $RH_p^d$ inequalities as $p$ tends to $1$ and show the sharp bound on $RH_1^d$ constant of the weight $w$ in terms of its $A_\infty^d$ constant.

We also take a look at the summation conditions of the Buckley type for the dyadic Reverse H\"{o}lder and Muckenhoupt weights and deduce them from an intrinsic lemma which gives a summation representation of the bumped average of a weight. Our lemmata also allow us to obtain summation conditions for continuous Reverse H\"{o}lder and Muckenhoupt classes of weights and both continuous and dyadic weak Reverse H\"{o}lder classes. In particular, it shows that a weight belongs to the class $RH_1$ if and only if it satisfies Buckley's inequality. We also show that the constant in each summation inequality of Buckley's type is comparable to the corresponding Muckenhoupt or Reverse H\"{o}lder constant.



To prove our main results we use the Bellman function technique.
}
\end{abstract}

\section{Definitions and Main Results.}
\label{s: intro}
Recently different approaches to dyadic and continuous $A_\infty$ class gave an essential improvement of the famous $A_2$ conjecture. The improvement, called $A_p-A_\infty$ bound for Calderon-Zygmund operators, was obtained by means of the observation that if a weight $w$ belongs to the Muckenhoupt class $A_p$, then it belong to a bigger class $A_\infty$, and a certain sequence satisfies the Carleson property. We refer the reader to papers \cite{HPTV}, \cite{HytPer2011} for the precise proof of $A_2-A_\infty$ bound (in \cite{HPTV} it is not formulated, but can be seen from the proof), and to \cite{HyLa} for a full proof of the $A_p-A_\infty$ bound.

Carleson sequences, related to $A_p$ weights, appeared in many papers, where boundedness of singular operators was studied. Many of them were proved using Bellman function method. Using this method, the Carleson embedding theorem was proved in \cite{NTV}. Results related to Carleson measures (partially proved with certain Bellman functions) also appeared in \cite{NTV-Tb}, \cite{Wittwer:00}, \cite{PetPot}. Also, the ``easy'' case of the two weight inequality, \cite{VV}, is a certain summation condition, and was also obtained by means of Bellman function. Most of our proofs will use very natural (but not totally sharp) Bellman functions.

Let us explain our results in more details.
In this paper we present equivalent definitions of Muckenhoupt classes $A_p$, Reverse H\"older classes $RH_p$, and prove sharp inequalities, that show that these definitions are indeed equivalent. One type of these definitions is given in terms of Carleson sequences. 
Also, we define limiting cases $A_\infty$ and $RH_1$, which in the continuous case appear to be same sets (see \cite{BR}), but in the dyadic case the class $RH_1$ is strictly bigger. We give equivalent definitions of these classes in terms of certain Carleson sequences; besides this, we give a sharp estimate on so called $A_\infty$ and $RH_1$ constants, which appears to be much harder than the continuous case (and, actually, somehow uses the continuous result).

The paper is organized as follows. We start by following paper \cite{BR}, with all the main definitions of dyadic Reverse H\"{o}lder and Muckenhoupt classes and state several equivalent ways define class $RH_1^d$. Also in Section \ref{s: intro} we state our first main result of the paper, Theorem \ref{theorem_RH1<Ainfty} in which we establish the comparability of dyadic $A_\infty^d$ and $RH_1^d$ constants.

In Section \ref{s: Buckley} we study summation conditions, introduced first in \cite{FeffermanKenigPipher:91} and \cite{Buckley:90}. Our second and third main results of this paper are, in fact, Lemma \ref{main_Lemma}  and Lemma \ref{inv_lemma}, two intrinsic lemmas from which we deduce Theorem \ref{our_Buckley_thm1} about comparability of sums in Buckley's summation condition and certain bumped averages of the weight $w$. Please note that even though Theorem \ref{our_Buckley_thm1} turns out an extremely strong fact and is very handy for H\"{o}lder and Muckenhoupt classes, our lemmata, especially Lemma \ref{main_Lemma} are much more general and could be applied to potentially large class of bumped averages of any nonnegative function $w$ and every interval $J \subset \R$. We show how Theorem \ref{our_Buckley_thm1} follows from our lemmata and how Buckley's theorem follows from Theorem \ref{our_Buckley_thm1}. It turns out that Theorem \ref{our_Buckley_thm1} is also sufficiently stronger than Buckley's theorem because it is not summation conditions for Reverse H\"{o}lder or Muckenhoupt classes, but comparability of averages and summations for any weight and any interval. This is illustrated in Theorem \ref{our_Buckley_thm}, where the comparability of constants in summation conditions and corresponding H\"{o}lder and Muckenhoupt constants of the weight is established in both continuous and dyadic cases.

In Section \ref{s: WRH} we talk about weak Reverse H\"{o}lder and Muckenhoupt classes. We start by giving definitions of these classes and state another consequence of Theorem \ref{our_Buckley_thm1}, Theorem \ref{weakineq}, which contains a version of Buckley's theorem but for the weak Reverse H\"{o}lder weights. The proof of Theorem \ref{weakineq} is essentially the same as the proof of theorem \ref{our_Buckley_thm1}, so we skip most of the details.

All Bellman function proofs can be found in Section \ref{s: proofs}. We start with proof of Lemma \ref{main_Lemma}, which we think is the simplest of three Bellman function proofs given in this paper and is a nice introduction to the Bellman function technique. Bellman function technique is not new, but as far as we know it is the first place where Bellman function technique is applied in such ``intrinsic'' setup. By ``intrinsic'' here we mean that lemma has function $A(x)$ as one of the parameters, convexity properties of function $A$ are then used to build Bellman function for the inequality. Proof of Lemma \ref{main_Lemma} is followed by the proof of Lemma \ref{inv_lemma} which we hope will be easy to digest after proof of Lemma \ref{main_Lemma}. Proof of Theorem \ref{theorem_RH1<Ainfty} is the hardest one and takes last eighteen pages of the paper. The proof itself is in fact very similar to the proof of continuous version of Theorem \ref{theorem_RH1<Ainfty}, which can be found in \cite{BR}. This dyadic proof is longer than he continuous one because in the dyadic case we have to deal with many details that are specific for the dyadic Bellman function proof in the non-convex domain. We encourage the reader to understand the proof of Theorem 1.1 from \cite{BR} first and then read our proof of Theorem \ref{theorem_RH1<Ainfty}.

All results of this paper are in one-dimensional case only.

{\bf{Acknowledgements}}

Authors are grateful to A. Volberg for useful suggestions in proving Theorem \ref{theorem_RH1<Ainfty} and to V. Vasyunin for useful discussions.

We would also like to express our gratitude to C. Thiele, I. Uriarte-Tuero and A. Volberg for organizing the Summer School 2010 in UCLA, where this paper was originated and C. P\'{e}rez and R. Esp\'{\i}nola for organizing the Summer School 2011 in Seville and AIM workshop, where we finished this paper.

\subsection{First Definitions} Let $\mathcal{D}$ be the dyadic grid $\mathcal{D}:= \left\{I\subset \R : I = [k 2^{-j}, (k+1) 2^{-j}); k, j \in \Z\right\}$.

We say that $w$ is a {\it weight} if it is a locally integrable function on the real line,
positive almost everywhere (with respect to the Lebesgue measure).
Let $\av{w}{J}$ be the average of a weight $w$ over a given interval $J \subset \R$:
$$
\av{w}{J} := \frac{1}{\mdl{J}} \int_J w\, dx
$$
and $\Delta_J w$ be defined by
$$
\Delta_J w := \av{w}{J^+} - \av{w}{J^-},
$$
where $J^+$ and $J^-$ are left and right dyadic children of the interval $J$.
\begin{defin}
A weight $w$ belongs to the dyadic \textit{Muckenhoupt class} $A_p^d$ whenever its dyadic Muckenhoupt constant $[w]_{A_p^d}$ is finite:
\begin{equation}\label{defAp}
[w]_{A_p^d} := \sup_{J\in \mathcal{D}} \; \av{w}{J} \av{w^{-\frac{1}{p-1}}}{J}^{p-1} < \infty.
\end{equation}
\end{defin}
\begin{zamech} \label{z: def Ap}
The inequality (\ref{defAp}) can be rewritten in the following way:
$$
0 \leqslant \av{w^{-\frac{1}{p-1}}}{J} - \av{w}{J}^{-\frac{1}{p-1}} \leqslant \left([w]_{A_p^d}^{\frac{1}{p-1} } -1 \right) \av{w}{J}^{-\frac{1}{p-1}}.
$$
We will use this way to write definitions of Reverse H\"{o}lder and Muckenhoupt classes later in the proof of Theorem \ref{our_Buckley_thm}.
\end{zamech}
Note that by H\"{o}lders inequality, $[{w}]_{A_p^d} \geqslant 1$ holds for all $1<p<\infty$, as
well as the following inclusion:
$$
if \;\;\; 1<p\leqslant q < \infty \;\;\; then \;\;\; A_p^d\; \subseteq
A_q^d, \; \; \; \; [{w}]_{A_q^d} \; \leqslant \; [{w}]_{A_p^d}.
$$

So, for $1<p<\infty$ Muckenhoupt classes $A_p^d$ form an increasing chain. There are two natural limits of it  - as $p$ approaches $1$ and as $p$ goes to $\infty$. We will be interested in the limiting case as $p \rightarrow \infty$, $A_\infty^d = \bigcup_{p>1} A_p^d$. There are several equivalent definitions of it, we will state one that we are going to use (the natural limit of $A_p^d$ conditions, that also defines the $A_\infty^d$ constant of the weight $w$), for other equivalent definitions see \cite{GarciaCuervaRubioDeFrancia:85}, \cite{Grafakos:03} or \cite{Stein:93}.
\begin{equation}\label{defAinfty}
w \in A_\infty^d  \;\;\;\;\;\;\Longleftrightarrow\;\;\;\;\;\;[w]_{A_\infty^d} :=\;\;\; \sup_{J\in \mathcal{D}} \;\;\av{w}{J} \;e^{-\av{\log w}{J}} \;\; < \infty,
\end{equation}
where $\log$ stands for the regular natural logarithm.

\begin{zamech} \label{z: def Aoo}
The inequality (\ref{defAinfty}) can be rewritten in the following way:
$$
0 \leqslant \log \av{w}{J} - \av{\log w}{J} \leqslant \log \, [w]_{A^d_\infty}.
$$
\end{zamech}

Note also that if a weight $w$ belongs to the Muckenhoupt class $A_p^d$ for some $p>1$, or, equivalently, to the class $A_\infty^d$, then $w$ has to be a dyadicaly doubling weight, i.e. its dyadic doubling constant ${\mathcal{D}}^d (w):= sup_{I\in D} \frac{\av{w}{F(I)}}{\av{w}{I}}$, where $F(I)$ stands for the dyadic parent of the interval $I$, has to be finite.

\begin{defin}A weight $w$ belongs to the {\it dyadic Reverse H\"{o}lder class} $RH_p^d$ ($1<p<\infty$) if
\begin{equation}\label{defRHp}
[w]_{RH_p^d}\;:=\;\sup_{J\in \mathcal{D}} \;\frac{ \av{w^p}{J}^{1/p}}{ \av{w}{J}} < \infty.
\end{equation}
\end{defin}

\begin{zamech} \label{z: def RHp}
The inequality (\ref{defRHp}) can be rewritten in the following way:
$$
0 \leqslant \av{w^{p}}{J} - \av{w}{J}^{p} \leqslant \left([w]_{RH_p^d}^{p } -1 \right) \av{w}{J}^{p}.
$$
\end{zamech}

Note that by H\"{o}lders inequality the dyadic Reverse H\"{o}lder classes satisfy:
$$
if \;\;\; 1<p\leqslant q<\infty, \; \; \; then \; \; \; RH_q^d \;
\subseteq \; RH_p^d \;\;\; and \;\;\; 1 \; \leqslant \; [w]_{RH_p^d}
\; \leqslant \; [w]_{RH_q^d},
$$
which is similar to the inclusion chain of the $A_p^d$ classes, except inclusion runs in the opposite direction. And similarly we can consider two limiting cases $RH_\infty^d$ (the smallest) and $RH_1^d$ (the largest). Same as in the case of Muckenhoupt classes we are more interested in the largest one, let us call it $RH_1^d := \bigcup_{p>1} RH_p^d$.

The natural limit as $p \rightarrow 1^+$ of the Reverse H\"{o}lder inequalities is the following condition, which we will take as a definition of the class $RH_1^d$:
\begin{equation}\label{defRH1}
w \in RH_1^d \;\;\; \Longleftrightarrow \;\;\; [w]_{RH_1^d} \; := \; \sup_{J\in \mathcal{D}}\left\langle\frac{w}{\av{w}{J}} \log \frac{w}{\av{w}{J}}\right\rangle_J\;\;<\;\infty,
\end{equation}
where $\log$ is a regular logarithm base $e$, which could be negative. Nevertheless, by the Jensen inequality $RH_1$ constant defined this way is always nonnegative.

The $RH_1^d$ constant of the weight $w$ is the natural limit of $RH_p^d$ constants in the sense that for every interval $I \in \mathcal{D}$
\begin{equation}\label{RH1aslimitRHp}
\av{\frac{w}{\av{w}{I}} \log \frac{w}{\av{w}{I}}}{I}\; = \; \lim_{p \rightarrow 1^+} \frac{p}{p-1} \log\; \frac{\av{w^p}{I}^\frac{1}{p}}{\av{w}{I}}\;
\end{equation}

We want to make one remark about this definition.
\begin{zamech} \label{z: def RH1}
The inequality \ref{defRH1} can be rewritten in the following way:
$$
\av{w\log(w)}{J}\leqslant \av{w}{J} \log \av{w}{J}+Q\av{w}{J} \;\;\; \forall J\in \mathcal{D}.
$$
Note that since function $x \log x$ is concave, by Jensen's inequality we also have
$$
\av{w}{J} \log \av{w}{J} \leqslant \av{w\log(w)}{J}.
$$
\end{zamech}



In the continuous case, for $A_\infty$ and $RH_1$ in 1974 Coifman and Fefferman showed that $A_\infty = \bigcup_{p>1} RH_p = RH_1$, in the dyadic case it is not true. One can only claim the inclusion $A_\infty^d \subset RH_1^d$. As for the other inclusion, it only holds for the dyadicaly doubling weights since, unlike the $A_p^d$ weights, dyadic Reverse H\"{o}lder weights do not have to be doubling. An example of such weight can be found in Buckley \cite{Buckley:90}.




{\bf Different ways to define $RH_1$ constant of the weight $w$.}\label{diffways}
 First, observe that, trivially, logarithm in the definition of the $RH_1$ constant can be replaced by $\log^+ (x)$,
$\left( \log^+ (x) = \max (\log x,0) \right)$ or $\log (e+x)$, which will, however, increase the $RH_1$ constant slightly.


Secondly, from the Stein lemma (see \cite{Stein:1969}), we know that
$$
3^{-n} \; \av{M(f \chi_I) }{I} \; \leqslant \; \av{ f \; \log \left( e + \frac{f}{\av{f}{I}} \right) }{I} \; \leqslant \; 2^n \; \av{M(f \chi_I)}{I}
$$
Thus an equivalent way to define $RH_1$ constant is
\begin{equation}
\label{def_equiv1_RH1} [w]_{RH_1^{d\prime}} \; := \; \sup \frac{1}{w(I)} \int_I M(w \chi_I)dx,
\end{equation}
which, indeed, is one of the ways to define class $A_\infty$, see for example \cite{Wilbook} or \cite{HytPer2011}.

One can also define dyadic Reverse H\"{o}lder and Muckenhoupt constants using Luxemburg norms. Same is true for $RH_1^d$-constant. Let us first define Luxemburg norm of a function in the following way: for an Orlitz function $\Phi: [0,\infty]\mapsto [0,\infty]$, we define $\norm{w}{\Phi(L),I}$ to be:
$$
\norm{w}{\Phi(L),I} \; := \; \inf \left\{ \lambda > 0\colon \; \frac{1}{|I|} \int_I \Phi\left( \frac{|w|}{\lambda} \right) \; \leqslant \; 1 \right\}.
$$
Iwaniec and Verde in \cite{IwVer} showed that for every $w$ and $I\subset \R^n$
$$
\norm{w}{L\log L,I} \; \leqslant \; \int_I w\log \left( e + \frac{w}{\av{w}{I}} \right) dx \; \leqslant \; 2 \norm{w}{L\log L,I},
$$
so another equivalent definition of the $RH_1$ constant of the weight $w$ is
\begin{equation}
\label{def_equiv2_RH1} [w]_{RH_1^{d\prime \prime}} \; := \; \sup_{I\in \mathcal{D}} \frac{\norm{w}{L\log L,I}}{\norm{w}{L,I}}.
\end{equation}

\subsection{First Main result of the paper}
In this section we carefully state the first result of the paper, and then explain other questions we study.

In fact, we prove the following sharp relationship between $RH_1^d$ and $A_\infty^d$ constants:
\begin{theorem}[\bf Main Result 1 : comparability of $RH_1$ and $A_\infty$ constants]
\label{theorem_RH1<Ainfty} If weight $w$ belongs to the Muckenhoupt class $A_\infty^d$, then $w \in RH_1^d$. Moreover,
\begin{equation}\label{maininequality}
[w]_{RH_1^d} \;\leqslant \; C \;  [w]_{A_\infty^d},
\end{equation}

where the constant $C$ can be taken to be $\log(16)$ ($C=\log(16)$). Moreover, the constant $C=\log(16)$ is the best possible.
\end{theorem}
Bellman function proof of this theorem can be found in Section \ref{s:proofH1<eAinfty}. An independent proof of the analogue of this theorem for the constant $[w]_{RH_1^{d\prime}}$ was recently independently obtained in \cite{HytPer2011}.

Note that all of the above is true in the continuous case and can be found in \cite{BR} (with sharp constant $C=e$, and with the double exponential lower bound). Note also that the lower bound (Theorem 1.2 in \cite{BR}) in dyadic case cannot possibly hold since the class $RH_1^d$ is strictly larger than $A_\infty^d$.


\section{Summation conditions on weights}

\label{s: Buckley}

In this section we will introduce and discuss an important set of inequalities that characterize the dyadic Reverse H\"{o}lder and Muckenhoupt classes. We are mostly interested in the dyadic results here, so we will follow Buckley \cite{Buckley:93_2}. Note that the inequalities we are going to discuss in this section have continuous analogues, and many facts and questions here apply to the continuous case as well (see \cite{FeffermanKenigPipher:91}).

As we discussed earlier, $RH_1^d \ne A_\infty^d$ because all dyadic Muckenhoupt conditions imply that the weight is dyadically doubling
, while dyadic Reverse H\"{o}lder conditions allow nondoubling weights (in the continuous case both Reverse H\"{o}lder and Muckenhoupt conditions imply continuous doubling property). For the dyadically doubling weights the $RH_1^d$ and $A_\infty^d$ conditions are equivalent.

We will now state a theorem that characterizes dyadic Reverse H\"{o}lder and Muckenhoupt classes via summation conditions.
We attribute this theorem to Buckley, however all parts but the Buckley's inequality (part $(2)$) in the continuous case and part $(4)$ in the dyadic case first appeared in \cite{FeffermanKenigPipher:91} and are due to Fefferman, Kenig and Pipher.

\begin{theorem} {\bf [Buckley'93]} \label{Buckley_theorem}

Suppose $1 < p < \infty$ and $w$ is a doubling weight. Then

(1) $w \in RH_p^d$ if and only if on every dyadic interval $J$

\begin{equation} \label{Buckley_RHp}
\frac{1}{\mdl{J}}\sum_{I \in \mathcal{D}(J)}{ \left(\frac{\Delta_I w}{\av{w}{I}}\right)^2 \av{w}{I}^p \mdl{I}}
\leq
K \av{w}{J}^p,
\end{equation}

moreover, $K \leq C [w]_{RH_p^{d}}^p$.

(2) (Buckley's inequality) $w \in RH_{1}^{d}$ if and only if for some $K>0$ on every dyadic interval $J$

\begin{equation} \label{Buckley_RH1}
\frac{1}{\mdl{J}}\sum_{I \in \mathcal{D}(J)}{ \left(\frac{\Delta_I w}{\av{w}{I}}\right)^2  \av{w}{I} \mdl{I}}
\leq
K\av{w}{J}.
\end{equation}

\bigskip

(3) $w \in A_{p}^{d}$ if and only if on every dyadic interval $J$

\begin{equation} \label{Buckley_Ap}
\frac{1}{\mdl{J}}\sum_{I \in \mathcal{D}(J)}{  \left(\frac{\Delta_I w }{\av{w}{I}}\right)^2  (\av{w}{I})^{-\frac{1}{p-1}} \mdl{I}}
\leq
K \av{w}{J}^{-\frac{1}{p-1}}.
\end{equation}


(4) (Fefferman - Kenig - Pipher inequality) $w \in A_{\infty}^{d}$ if and only if on every dyadic interval $J$

\begin{equation} \label{Buckley_Ainfty}
\frac{1}{\mdl{J}}\sum_{I \in \mathcal{D}(J)}{  \left(\frac{\Delta_I w }{\av{w}{I}}\right)^2  \mdl{I}}
\leq
C \log [w]_{A_{\infty}^{d}}.
\end{equation}

\end{theorem}

The Buckley's inequality (part $(2)$) is the one we are mostly interested in since as we will see later it characterizes class $RH_1^d$; it is also the only one stated without the sharp constant. In \cite{Wittwer:00} Wittwer showed that in the case $w \in A_2^d$ Buckley's inequality holds with $K = C [w]_{A_2^d}$ and this linear dependence on the $A_2^d$ constant of the weight $w$ is sharp, which is the best known result for Buckley's inequality.  Also, in the Fefferman-Kenig-Pipher inequality the sharp constant is $C=8$, it was obtained by Vasyunin using the Bellman function method in \cite{Vasyunin:10}.

 Using the method of Bellman functions we are going to show that in Buckley's inequality  $K \leq C [w]_{RH_{1}^{d}}$. We also show that the assumption that $w$ is a doubling weight can be dropped. Finally, we show that the above four sums also satisfy the lower bound estimates in terms of the corresponding constants. Let us state our second main result in this paper now.

 We start with the following lemma, from which Theorem \ref{Buckley_theorem} will follow.

 \begin{lemma}\label{main_Lemma}
 Let $A(x)$ be a convex twice differentiable function on $(0,\infty)$ such that for all numbers $x$ and $t$, such that $x, x\pm t$ are in the domain of $A$, the following inequality holds:
 \begin{equation}\label{condition_A'''}
 A(x)-\frac{A(x-t)+A(x+t)}{2} + \alpha \, t^2 A''(x) \geqslant 0,
 \end{equation}
 with some constant $\alpha>0$ independent of $x$ and $t$.
 Then for every weight $w$ and an interval $J$ the following inequality holds:
 \begin{equation}\label{A_Buckley}
 \frac{1}{|J|} \sum_{I\in \mathcal{D}(J)} \left(\Delta_I w\right)^2 A^{\prime\prime}(\av{w}{I}) |I| \leqslant C  \left(\av{A(w)}{J} - A(\av{w}{J})\right).
 \end{equation}

 Moreover, if the second derivative of $A$ satisfies the following inequality for every $x \in (0, \infty)$ and every $\varepsilon \geqslant 0$
  \begin{equation}
  \label{condition_A''}
  \int_{-1}^{1} (1-|t|) \; A^{\prime\prime} (x + \varepsilon t) \; dt \; \geqslant \; q \; A^{\prime\prime}(x)
  \end{equation}
 with some positive constant $q$ uniformly on $x$ and $\varepsilon$, then  
the inequality (\ref{A_Buckley}) holds with constant $C = 8 \frac{1}{q}$.

 \end{lemma}

 The Bellman function proof of the Lemma \ref{main_Lemma} can be found in Section \ref{s: proof main lemma}.




\begin{zamech}
Note that if the second derivative of $A$ is a monotone function (\ref{condition_A''}) holds trivially with constant $q = \frac{1}{2}$, which makes Lemma \ref{main_Lemma} applicable to a large class of functions producing a number of new inequalities of Buckley's type. In particular, function $A(x)$ can be taken $A(x)= x^p, \; p>1$, $A(x) = x\log x$, $A(x)= x^{-\frac{1}{p-1}}, \; p>1$ or $A(x) = \log x$. In what follows we will see how these choices of the function $A(x)$ imply Buckley's theorem.
\end{zamech}

Now we want to introduce the ``reverse'' lemma, which is true for particular (most interesting for us) choices of the function $A$.
\begin{lemma}\label{inv_lemma}

(1) Let $A(x)$ be a function, defined on $(0,\infty)$ such that
\begin{equation}\label{A_reverse}
A(x)-\frac{A(x-t)+A(x+t)}{2}+\beta \, t^2 A''(x) \; \geqslant \; 0,
\end{equation}
holds with some positive constant $\beta$ independent of $x$ and $t$.

Then for every weight $w$ and an interval $J$
 \begin{equation}
 \frac{1}{|J|} \sum_{I\in \mathcal{D}(J)} \left(\Delta_I w\right)^2 \, A^{\prime\prime}(\av{w}{I}) \; |I| \; \geqslant \; C \left(\av{A(w)}{J} - A(\av{w}{J})\right).
 \end{equation}

Moreover, condition (\ref{A_reverse}) holds for functions $A(x)=x^p$, for all $p>1$ and for $A(x)=x\log x$.

(2) If $A$ satisfies the inequality
\begin{equation}\label{A_reverse_db}
A(x)-\frac{A(x-t)+A(x+t)}{2}+\beta \, t^2 A''(x) \; \geqslant \; 0
\end{equation}
whenever $0<t<\frac{C-1}{C}x$ (for $C>1$; $\beta$ depends on $C$). Then for every doubling weight $w$ and an interval $J$
 \begin{equation}
 \frac{1}{|J|} \sum_{I\in \mathcal{D}(J)} \left(\Delta_I w\right)^2 A^{\prime\prime}(\av{w}{I}) \; |I| \; \geqslant \; C \left(\av{A(w)}{J} - A(\av{w}{J})\right),
 \end{equation}
where the constant $C$ depends on the doubling constant of $w$.

Moreover, condition (\ref{A_reverse_db}) holds for functions $A(x)=x^{-\frac{1}{p-1}}$ for all $p>1$ and for $A(x)=-\log(x)$.
\end{lemma}
Bellman Function proof of Lemma \ref{inv_lemma} can be found in section \ref{s: proof inv lemma}.

\begin{zamech}
Note that in Lemma \ref{inv_lemma}, similarly to Lemma \ref{main_Lemma}, we can also write conditions (\ref{A_reverse}) and (\ref{A_reverse_db}) in the integral form, but in this case (\ref{A_reverse}) and (\ref{A_reverse_db}) are easier to check at least for the functions we are interested in.
\end{zamech}

From our lemmata, by taking $A(x) = x^p$ and $A(x) = x^{-\frac{1}{p-1}}$\; $p>1$, $A(x)= x \log x$ and  $A(x)=\log(x)$,  we derive the following theorem.

\begin{theorem} [\bf Main result 2 : Representation of bumped averages] \label{our_Buckley_thm1}

Suppose $1 < p < \infty$ and $w$ is weight. Then

(1){\bf(case $A(x)=x^p$, $p>1$)} There are real positive  constants $c$ and $C$ independent of the weight $w$, such that for every interval $J$

\begin{equation} \label{Buckley_RHp_1}
c (\av{w^p}{J} - \av{w}{J}^p)
\leqslant
\frac{1}{\mdl{J}}\sum_{I \in \mathcal{D}(J)}{ \left(\frac{\Delta_I w}{\av{w}{I}}\right)^2 \av{w}{I}^p \mdl{I}}
\leqslant
C (\av{w^p}{J} - \av{w}{J}^p).
\end{equation}

(2) {\bf(case $A(x)= x \log x$)}  There are real positive  constants $c$ and $C$ independent of the weight $w$, such that for every interval $J$

\begin{equation} \label{Buckley_RH1_1}
c \left( \av{w \log w}{J} - \av{w}{J} \log \av{w}{J}\right)
\leqslant
\frac{1}{\mdl{J}}\sum_{I \in \mathcal{D}(J)}{ \left(\frac{\Delta_I w}{\av{w}{I}}\right)^2  \av{w}{I} \mdl{I}}
\leq
C \left(\av{w \log w}{J} - \av{w}{J} \log \av{w}{J}\right).
\end{equation}

\bigskip

(3)  {\bf (case $A(x)= x^{-\frac{1}{p-1}}$)}  There is a real positive  constant $C$ independent of $w$, such that for every interval $J$

\begin{equation} \label{Buckley_Ap_1}
\frac{1}{\mdl{J}}\sum_{I \in \mathcal{D}(J)}{  \left(\frac{\Delta_I w}{\av{w}{I}}\right)^2  \av{w}{I}^{-\frac{1}{p-1}} \mdl{I}}
\leq
C  \left(\av{w^{-\frac{1}{p-1}}}{J} - \av{w}{J}^{-\frac{1}{p-1}}\right).
\end{equation}

Moreover, if $w$ is a doubling weight, then  there exists constant $c$ that may depend on the doubling constant of the weight $w$, such that  for every interval $J$

\begin{equation} \label{Buckley_Ap_1}
c \left(\av{w^{-\frac{1}{p-1}}}{J} - \av{w}{J}^{-\frac{1}{p-1}}\right)
\leqslant
\frac{1}{\mdl{J}}\sum_{I \in \mathcal{D}(J)}{  \left(\frac{\Delta_I w}{\av{w}{I}}\right)^2  \av{w}{I}^{-\frac{1}{p-1}} \mdl{I}}.
\end{equation}

(4) {\bf(case $A(x)= - \log x$)}  There is a real positive  constant $C$ independent of $w$, such that for every interval $J$

\begin{equation} \label{Buckley_Ainfty_1}
\frac{1}{\mdl{J}}\sum_{I \in \mathcal{D}(J)}{  \left(\frac{\Delta_I w}{\av{w}{I}}\right)^2  \mdl{I}}
\leqslant
C (\log \av{w}{J} - \av{\log w}{J}).
\end{equation}

Moreover, if $w$ is a doubling weight, then  there exists constant $c$ that may depend on the doubling constant of the weight $w$, such that  for every interval $J$

\begin{equation} \label{Buckley_Ainfty_1}
c (\log \av{w}{J} - \av{\log w}{J})
\leqslant
\frac{1}{\mdl{J}}\sum_{I \in \mathcal{D}(J)}{  \left(\frac{\Delta_I w}{\av{w}{I}}\right)^2  \mdl{I}}.
\end{equation}

\end{theorem}

Theorem \ref{our_Buckley_thm1} immediately follows from Lemma \ref{main_Lemma}, the Remark after it and Lemma \ref{inv_lemma}. We will leave its proof to the reader. Instead, let us show how Theorem \ref{our_Buckley_thm1} implies Buckley's theorem in the dyadic and continuous cases and in the case of weak Reverse H\"{o}lder classes.

In order to write our results in a more compact way we will start by introducing another way to define Reverse H\"{o}lder and Muckenhout constants. We will call them Buckley's constants and denote by $[w]_{RH_p^{d, B}}$  and $[w]_{A_p^{d, B}}$:

$$
[w]_{RH_p^{d, B}} := \inf \left\{ Q>1 \;\;s.t.\;\;\forall J \in \mathcal{D} \;\;\;\; \frac{1}{\mdl{J}}\sum_{I \in \mathcal{D}(J)}{ \left(\frac{\Delta_I w}{\av{w}{I}}\right)^2 \av{w}{I}^p \mdl{I}}  \leqslant Q \av{w}{J}^p\right\} ,\;\;\;\;p\geqslant 1,
$$

and similarly we can define continuous Buckley's Reverse H\"{o}lder constants

$$
[w]_{RH_p^{ B}} := \inf \left\{ Q>1 \;\;s.t.\;\;\forall J \subset \R \;\;\;\; \frac{1}{\mdl{J}}\sum_{I \in \mathcal{D}(J)}{ \left(\frac{\Delta_I w}{\av{w}{I}}\right)^2 \av{w}{I}^p \mdl{I}}  \leqslant Q \av{w}{J}^p\right\} ,\;\;\;\;p\geqslant 1.
$$
And similarly  for $1<p < \infty$ we define dyadic and continuous Buckley's Muckenhoupt constants:
$$
[w]_{A_p^{d, B}} := \inf \left\{ Q>0 \;\;s.t.\;\;\forall J \in \mathcal{D} \;\;\;\; \frac{1}{\mdl{J}}\sum_{I \in \mathcal{D}(J)}{  \left(\frac{\Delta_I w}{\av{w}{I}}\right)^2  \av{w}{I}^{-\frac{1}{p-1}}\mdl{I}}
\leqslant   Q \av{w}{J}^{-\frac{1}{p-1}}\right\} ,
$$
and
$$
[w]_{A_p^{ B}} := \inf \left\{ Q >0 \;\;s.t.\;\;\forall J \subset \R \;\;\;\; \frac{1}{\mdl{J}}\sum_{I \in \mathcal{D}(J)}{  \left(\frac{\Delta_I w}{\av{w}{I}}\right)^2 \av{w}{I}^{-\frac{1}{p-1}} \mdl{I}}
\leqslant   Q \av{w}{J}^{-\frac{1}{p-1}}\right\}.
$$
and in the $A_\infty$ case we have
$$
[w]_{A_\infty^{d, B}} := \inf \left\{ Q > 0 \;\;s.t.\;\;\forall J \in \mathcal{D} \;\;\;\; \frac{1}{\mdl{J}}\sum_{I \in \mathcal{D}(J)}{  \left(\frac{\Delta_I w}{\av{w}{I}}\right)^2  \mdl{I}} \;\;
\leqslant \;\;  Q \;\; \right\} ,
$$
and
$$
[w]_{A_\infty^{ B}} := \inf \left\{ Q > 0 \;\;s.t.\;\;\forall J \subset \R \;\;\;\; \frac{1}{\mdl{J}}\sum_{I \in \mathcal{D}(J)}{  \left(\frac{\Delta_I w}{\av{w}{I}}\right)^2  \mdl{I}}\;\;
\leqslant   \;\; Q \;\;\right\}.
$$

Note that in the Reverse H\"{o}lder case in Buckley's constants we do not need to define separately the $RH_1$ constants.
We are ready to state the result about comparability of Buckley's constants to the regular Reverse H\"{o}lder and Muckenhoupt constants.

\begin{theorem} [\bf \; Main\; result\; 2 : comparability\; of\; constants\; in\; summation\; conditions] \label{our_Buckley_thm}
\;\quad (1) Suppose $1 < p < \infty$ then there are positive constants $C$ and $c$ such that for every weight $w$
$$
c \;([w]_{RH_p^{d}}^{p} - 1) \; \leqslant\;  [w]_{RH_p^{d , B}} \; \leqslant \; C \; ([w]_{RH_p^{d}}^{{p}} - 1)
$$
and
$$
c \;([w]_{RH_p}^{p} - 1) \; \leqslant\;  [w]_{RH_p^{ B}} \; \leqslant \; C \; ([w]_{RH_p}^{{p}} - 1)
$$

(2) In the case  $p=1$ there are positive constants $C$ and $c$ such that for every weight $w$
$$
c [w]_{RH_1^{d}} \leqslant [w]_{RH_1^{d, B}} \leqslant C [w]_{RH_1^{d}}
$$
and
$$
c [w]_{RH_1} \leqslant [w]_{RH_1^{B}} \leqslant C [w]_{RH_1}.
$$

\bigskip
(3) For any $1<p<\infty$ there is a positive constants $C$ such that for every weight $w$
$$
[w]_{A_p^{d , B}} \leqslant \; C \; ([w]_{A_p^{d}}^{\frac{1}{p-1}} -1)
\;\;\;\;
and
\;\;\;\;
[w]_{A_p^{ B}} \leqslant C ([w]_{A_p}^{\frac{1}{p-1}} -1)
$$

(4) In the case $p=\infty$ there is a positive constants $C$ such that for every weight $w$
$$
[w]_{A_\infty^{d, B}} \leqslant \; C \; \log {[w]_{A_\infty^{d}}}
\;\;\;\;
and
\;\;\;\;
[w]_{A_\infty^{B}} \leqslant \; C \; \log{[w]_{A_\infty}}.
$$

\bigskip

Moreover, if $w$ is a doubling weight then

(5) For any $1<p<\infty$
$$
 c_d\;  ([w]_{A_p^{d}}^\frac{1}{p-1} - 1) \leqslant [w]_{A_p^{d , B}} \;\;\;\; and \;\;\;\; c\;  ([w]_{A_p}^\frac{1}{p-1} - 1) \leqslant [w]_{A_p^{ B}}
$$
holds with positive constants $c_d$ and $c$ that depend on the (dyadic) doubling constant of the weight $w$.

(6) In the case $p=\infty$
$$
 c_d\; \log {[w]_{A_\infty^{d}}} \leqslant  [w]_{A_\infty^{d , B}} \;\;\;\;and\;\;\;\;c\; \log {[w]_{A_\infty}} \leqslant [w]_{A_\infty^B}
$$
holds with positive constants $c_d$ and $c$ that depend on the (dyadic) doubling constant of the weight $w$.

\end{theorem}


We now show how Theorem \ref{our_Buckley_thm} follows from the Theorem \ref{our_Buckley_thm1}. Note also that in parts (5) and (6) of the Theorem \ref{our_Buckley_thm} constant $c_d$ and $c$ are different because one depends on the dyadic doubling constant of the weight $w$ and the other one depends on the continuous doubling constant of $w$.

\begin{proof}
We will prove case (1), all other cases are proved in a similar way with only minor changes and will be left to the reader.

We will show that the first part of Theorem \ref{our_Buckley_thm} follows from the first part of Theorem \ref{our_Buckley_thm1}, from which we know that there are constants $c$ and $C$ such that for any weight $w$ and interval $J \subset \R$
\begin{equation}
\label{f: RHp bump sum}
c (\av{w^p}{J} - \av{w}{J}^p)
\leqslant
\frac{1}{\mdl{J}}\sum_{I \in \mathcal{D}(J)}{ \left(\frac{\Delta_I w}{\av{w}{I}}\right)^2 \av{w}{I}^p \mdl{I}}
\leqslant
C (\av{w^p}{J} - \av{w}{J}^p).
\end{equation}

First, we assume that $w \in RH_p^{(d)}$ (dyadic or continuous), which means, by Remark \ref{z: def RHp},that for every (dyadic) interval $J\subset \R$
$$
0 \leqslant \av{w^p}{J} - \av{w}{J}^p \leqslant ([w]_{RH_p^{(d)}}^p-1) \av{w}{J}^p
$$
So, by inequality \ref{f: RHp bump sum} we have that
$$
\frac{1}{\mdl{J}}\sum_{I \in \mathcal{D}(J)}{ \left(\frac{\Delta_I w}{\av{w}{I}}\right)^2 \av{w}{I}^p \mdl{I}}
\leqslant \;
C \; ([w]_{RH_p^{(d)}}^p-1) \av{w}{J}^p,
$$
hence $[w]_{RH_p^{(d), B}} \, \leqslant \, C \; ([w]_{RH_p^{(d)}}^p-1)$.

Second, assume that $w \in RH_p^{(d), B}$, so for each (dyadic) interval $J \subset \R$
$$
\frac{1}{\mdl{J}}\sum_{I \in \mathcal{D}(J)}{ \left(\frac{\Delta_I w}{\av{w}{I}}\right)^2 \av{w}{I}^p \mdl{I}}
\leqslant \;
 [w]_{RH_p^{(d), B}} \av{w}{J}^p.
$$

Then from (\ref{f: RHp bump sum}) we deduce that
$$
\av{w^p}{J} - \av{w}{J}^p \;\leqslant \; \frac{1}{c} \; [w]_{RH_p^{(d), B}} \av{w}{J}^p,
$$
which means that $w \in RH_p^{(d)}$ and $c \; ([w]_{RH_p^{(d)}}^p -1) \leqslant [w]_{RH_p^{(d), B}}$.

Parts (2), (3) and (4) of the Theorem \ref{our_Buckley_thm} are proved in exactly the same way, using Remarks \ref{z: def RH1}, \ref{z: def Ap} and \ref{z: def Aoo} and the corresponding parts of Theorem \ref{our_Buckley_thm1}. The doubling assumptions in (3) and (4) also come from the Theorem \ref{our_Buckley_thm1}.
\end{proof}

The Theorem \ref{our_Buckley_thm} obviously implies Buckley's theorem (Theorem \ref{Buckley_theorem}), but our Theorem \ref{our_Buckley_thm1} is even stronger then this. Since Theorem \ref{our_Buckley_thm1} shows comparability of summations for a given weight with its bumped averages, we can also write summation conditions for the weak Reverse H\"{o}lder classes in the similar way.

\section{Summation Conditions for the Weak Reverse H\"older Classes}
\label{s: WRH}

In this section we discuss the weak Reverse H\"older class $RHW_p$, $p\geqslant 1$. We remind that the definition of the $RHW_p$-constant. For simplicity, we drop the superscript $d$, that referred to dyadic case.

All of the above is true in the continuous case as well, when all suprema are taken over any interval $J \subset \R$. We won't repeat all the definitions but will refer the reader to \cite{BR}.

We also give the definition of so called ``weak'' reverse H\"older class $RHW^d_p$.
\begin{defin}
In the dyadic case let $J^{\star}$ stand for the dyadic parent of $J \in D$. Then weight $w$ belongs to the dyadic weak Reverse H\"{o}lder class $RHW_p^d$, $p>1$, if and only if its weak Reverse H\"{o}lder constant is finite:
\begin{equation}
\label{weakrhp}
w \in RHW_p^d \;\;\;\;\Longleftrightarrow \;\;\;\; [w]_{RHW_p^d} := \sup_{J\in \mathcal{D}} \frac{\av{w^p}{J}^{\frac{1}{p}}}{\,\av{w}{J^\star}} <\infty.
\end{equation}
For $p=1$ we define the $RHW_1^d$ class as follows:
\begin{equation}\label{weakrh1}
w \in RHW_1^d \;\;\;\;\Longleftrightarrow \;\;\;\; [w]_{RHW_1}^d := \sup_{J\in \mathcal{D}} \av{\frac{w}{\;\;\av{w}{J^\star}}\log\frac{w}{\;\av{w}{J}}}{J}<\infty.
\end{equation}

In the continuous case, for any interval $J\subset \R$ let $2 J$ stand for the interval concentric with $J$ of the length twice the length of interval $J$. Then weak Reverse H\"{o}lder classes $RHW_p, p>1$ and $RHW_1$ are defined by
\begin{equation}
\label{d: RHWp}
w \in RHW_p \;\;\;\;\Longleftrightarrow \;\;\;\; [w]_{RHW_p} := \sup_{J\subset \R} \frac{\av{w^p}{J}^{\frac{1}{p}}}{\,\av{w}{2J}} <\infty.
\end{equation}
and
\begin{equation}\label{d: RHW1}
w \in RHW_1\;\;\;\;\Longleftrightarrow \;\;\;\; [w]_{RHW_1}^d := \sup_{J\subset \R} \av{\frac{w}{\;\;\av{w}{2J}}\log\frac{w}{\;\av{w}{J}}}{J}<\infty.
\end{equation}
\end{defin}
We again note that it is important that in the definitions of $RHW_1^d$ and $RHW_1$ inside the $\log$ we divide by the average of $w$ over the interval $J$, not by the average over $J^\star$ or $2J$.

\begin{zamech}
We now explain why the definition of the weak Reverse H\"older constant, \eqref{weakrh1}, makes sense. In fact, in spirit of the formula above, we can define it as
$$
[w]_{RHW_1^{d\prime \prime}} \; := \; \sup_{I\in \mathcal{D}} \frac{\norm{w}{L\log L,I}}{\norm{w}{L,I^\star}}.
$$
\end{zamech}

\begin{zamech}
Also note that as in the strong case we can rewrite (\ref{weakrhp}), (\ref{weakrh1}), (\ref{d: RHWp}) and (\ref{d: RHW1}) as:
\begin{eqnarray}
\label{d: WRHpd e}
w \in RHW_p^d \;\;\;\;&\Longleftrightarrow& \;\;\;\; 0 \leqslant \av{w^p}{J} \leqslant [w]_{RHW_p^d}^p \av{w}{J^{\star}}^p \;\;\;\;\; \forall J \in \mathcal{D},\\
\label{d: RHW1d e}
w \in RHW_1^d \;\;\;\;&\Longleftrightarrow& \;\;\;\; 0\leqslant \av{w \log w}{J} - \av{w}{J} \log \av{w}{J} \leqslant [w]_{RHW_1^d} \av{w}{J^\star} \;\;\;\;\; \forall J \in \mathcal{D},\\
\label{d: WRHp e}
w \in RHW_p \;\;\;\;&\Longleftrightarrow& \;\;\;\; 0 \leqslant \av{w^p}{J} \leqslant [w]_{RHW_p}^p \av{w}{2J}^p \;\;\;\;\; \forall J \subset \R,\\
\label{d: RHW1 e}
w \in RHW_1^d \;\;\;\;&\Longleftrightarrow& \;\;\;\; 0\leqslant \av{w \log w}{J} - \av{w}{J} \log \av{w}{J} \leqslant [w]_{RHW_1^d} \av{w}{2J} \;\;\;\;\; \forall J \subset \R.
\end{eqnarray}
\end{zamech}

\begin{defin} We are ready to define dyadic and continuous weak Buckley Reverse H\"{o}lder constants now in the most natural way. For any $p\geqslant 1$ let
$$
[w]_{RHW_p^{d, B}} := \inf \left\{ Q>0 \;\;s.t.\;\;\forall J \in \mathcal{D} \;\;\;\; \frac{1}{\mdl{J}}\sum_{I \in \mathcal{D}(J)}{ \left(\frac{\Delta_I w}{\av{w}{I}}\right)^2 \av{w}{I}^p \mdl{I}}  \leqslant Q \av{w}{J^\star}^p\right\}
$$
and
$$
[w]_{RHW_p^{B}} := \inf \left\{ Q>0 \;\;s.t.\;\;\forall J \subset \R \;\;\;\; \frac{1}{\mdl{J}}\sum_{I \in \mathcal{D}(J)}{ \left(\frac{\Delta_I w}{\av{w}{I}}\right)^2 \av{w}{I}^p \mdl{I}}  \leqslant Q \av{w}{J^\star}^p\right\}.
$$
\end{defin}




Now we are ready to state the following theorem, which is also a consequence of the Theorem \ref{Buckley_Ainfty_1}.
\begin{theorem}\label{weakineq}
A weight $w$ belongs to $RHW_p^d$ if and only if the weak Buckley constant, $[w]_{RHW_p^{d, B}}$, is finite. Moreover, there exist positive constants $C_1$ that does not depend on $w$ and $p$ and $C_2$ that may depend on $p$, such that for any $p>1$
$$
[w]_{RHW_p^{d,B}} \leqslant C_1 [w]_{RHW_p^{d}}^p \; \; \; and \;\;\;\; [w]_{RHW_p^{d}} \leqslant C_2 ([w]_{RHW_p^{d, B}} + 1)^\frac{1}{p}
$$
and same in the continuous case
$$
[w]_{RHW_p^{B}} \leqslant C_1 [w]_{RHW_p}^p \; \; \; and \;\;\;\; [w]_{RHW_p} \leqslant C_2 ([w]_{RHW_p^{B}} + 1)^\frac{1}{p}.
$$

In the case  $p=1$ there are positive constants $C$ and $c$ such that
$$
c [w]_{RH_1^{d, B}} \leqslant [w]_{RH_1^{d}} \leqslant C [w]_{RH_1^{d, B}}
$$
and
$$
c [w]_{RH_1} \leqslant [w]_{RH_1^B} \leqslant C [w]_{RH_1}.
$$
 \end{theorem}
\begin{proof}
Proofs for continuous and dyadic cases are identical, so we will do both continuous and dyadic cases simultaneously.

For $p>1$, by Theorem \ref{our_Buckley_thm1}, part (1) we know that for any weight $w$ and any interval $J$  the following holds:
\begin{equation} \label{Buckley_RHp_2}
c (\av{w^p}{J} - \av{w}{J}^p)
\leqslant
\frac{1}{\mdl{J}}\sum_{I \in \mathcal{D}(J)}{ \left(\frac{\Delta_I w}{\av{w}{I}}\right)^2 \av{w}{I}^p \mdl{I}}
\leqslant
C (\av{w^p}{J} - \av{w}{J}^p).
\end{equation}
Note that $\av{w}{J}$ is nonnegative, so if $w$ belongs to the (dyadic or continuous) class $RHW_p^{d}$ by (\ref{d: WRHpd e}) or (\ref{d: WRHp e}) we have that for every (dyadic) interval $J \subset \R$
\begin{equation} \nn
\frac{1}{\mdl{J}}\sum_{I \in \mathcal{D}(J)}{ \left(\frac{\Delta_I w}{\av{w}{I}}\right)^2 \av{w}{I}^p \mdl{I}}
\leqslant
C \av{w^p}{J} \leqslant C [w]_{RHW_p^{(d)}}^p \av{w}{F(J)}^p,
\end{equation}
where $F(J)$ is either the dyadic parent of $J$ or $2J$. So $[w]_{RHW_p^{(d), B}} \leqslant C [w]_{RHW_p^{(d)}}^p$.  To prove the reverse inequality we assume that $w$ is in (dyadic or continuous) $RHW_p^{(d),B}$, then
\begin{equation} \nn
c (\av{w^p}{J} - \av{w}{J}^p)
\leqslant
\frac{1}{\mdl{J}}\sum_{I \in \mathcal{D}(J)}{ \left(\frac{\Delta_I w}{\av{w}{I}}\right)^2 \av{w}{I}^p \mdl{I}}
\leqslant
[w]_{RHW_p^{(d),B}} \av{w}{F(J)}^p,
\end{equation}
from which we conclude that $\av{w^p}{J} \leqslant \frac{1}{c} [w]_{RHW_p^{(d),B}} \av{w}{F(J)}^p +  \av{w}{J}^p$. Note that $\av{w}{J} \leqslant 2 \av{w}{F(J)}$, so
$$
\av{w^p}{J} \leqslant \left(\frac{1}{c} [w]_{RHW_p^{(d),B}} +2^p\right) \av{w}{F(J)}^p,
$$
Which implies that $[w]_{RHW_p^{(d)}} \leqslant \left(\frac{1}{c} [w]_{RHW_p^{(d),B}} +2^p\right)^\frac{1}{p}$ and completes the proof of the theorem for $p>1$.


For $p=1$ we use the comparability (part (2) of Theorem \ref{our_Buckley_thm1}):
\begin{equation} \nn
c \left( \av{w \log w}{J} - \av{w}{J} \log \av{w}{J}\right)
\leqslant
\frac{1}{\mdl{J}}\sum_{I \in \mathcal{D}(J)}{ \left(\frac{\Delta_I w}{\av{w}{I}}\right)^2  \av{w}{I} \mdl{I}}
\leq
C \left(\av{w \log w}{J} - \av{w}{J} \log \av{w}{J}\right)
\end{equation}
together with the definitions of dyadic and continuous $RHW_1^{(d)}$ (\ref{d: RHW1d e}) and (\ref{d: RHW1 e}). This proof is similar to the continuous case is left to the reader.
\end{proof}
\begin{zamech}
We notice that we have proved the theorem for any pairs $(J, F(J))$, which satisfy the following two conditions:

(1) $J\subset F(J)$, and

(2) $|J|\geqslant c |F(J)|$.
\end{zamech}
\section{Bellman Function Proofs}
\label{s: proofs}

\subsection{Some history}
We now proceed to the Bellman-type proofs. Before we do it, we would like to make some historical overview.

Bellman function, related to investigation of weights by their own (i.e., not related to linear operators in weighted spaces), has been exploit in different papers.
Such properties as Reverse H\"{o}lder, $L^p$ estimates and distribution functions of $A_p$ weights were investigated in works ~\cite{Va}, ~\cite{Re}, ~\cite{DiWa}. In all these works the Bellman function was found for continuous $A_p$. We strongly refer the curious reader to these papers, since the search for Bellman function and extremal examples are given there in details.

Aside from these three papers, the theory of $BMO$ weights was developed in ~\cite{SlVa}. In this paper, together with the continuous $BMO$, authors considered the dyadic one. The dyadic problem appeared to be much more delicate in some sense, and required a lot of additional calculations. In what follows, we use several parts of the dyadic proof from ~\cite{SlVa}. It appears that in our case the same steps give the proof. However, some parts of our proof are more delicate.

We now point out two difficulties that we have. First of all, functions in ~\cite{SlVa} were absolutely explicit. In our case, as the reader will see, many ingredients are given implicitly, which makes things a little more complicated.

The main difficulty, though, is not the fact that we have implicit functions. In ~\cite{SlVa} authors noticed that the domain of their Bellman function has the following property: it can be enlarged, with a good estimate of this ``enlargement'', such that if endpoints and center of some interval are in the smaller domain, then the whole interval is in the enlarged domain.
It immediately implied that, if we do not care about sharp constants, we can get some nice estimates in dyadic case immediately from the continuous case.

In the Remark \ref{cannotenlarge} we prove that the domain of our Bellman function does not have this property. Therefore, without additional investigation, we can not make any dyadic statements. This means that we are ``forced'' to care about best constants and run a variant of the proof from ~\cite{SlVa}.

We also refer the reader to another dyadic problem, ~\cite{VaVo}. Authors obtained the exact Bellman function too. However, the domain of their function was convex, and, therefore, obviously had the above property.

In ~\cite{BR} authors introduced a certain function of two variables, that allows to prove the continuous case of the inequality.

We sketch the definition and application of this function and discuss the main difficulty of the dyadic problem.

We will start with the Bellman function proofs of the summation conditions (inequalities \ref{Buckley_Ainfty} - \ref{Buckley_RHp}) since they are simpler then the prove of theorem \ref{theorem_RH1<Ainfty}, which is much harder.

\subsection{Technical details}
In this section we want to prove the following proposition.
\begin{prop}
\begin{enumerate}
\item For any monotone non negative function $f(x)$ the following inequality holds for some constant $C$:
$$
 \int_{-1}^{1} (1-|t|)  f(x + \varepsilon t) dt \geqslant C f(x).
$$
\item If $A(x)$ satisfies
$$
 \int_{-1}^{1} (1-|t|) A^{\prime\prime} (x + \varepsilon t) dt \geqslant q A^{\prime\prime}(x)
 $$
then for some $\alpha>0$
$$
A(x)-\frac{A(x-t)+A(x+t)}{2}+\alpha t^2 A''(x)\leqslant 0.
$$
\item If $A(x)=x^p$, $p>1$, then for some $\beta>0$
$$
A(x)-\frac{A(x-t)+A(x+t)}{2}+\beta t^2 A''(x)\geqslant 0
$$
\item Let $C>1$ and $A(x)=x^{-\frac{1}{p-1}}$. Then there exists an $\alpha$, depending only on $p$ and $C$, such that the following inequality holds for any $t$, $0<t<\frac{C-1}{C}x$:
$$
A(x)-\frac{A(x+t)+A(x-t)}{2} +\beta t^2 A''(x) \geqslant 0.
$$
Moreover, one can take
$$
\beta = \frac{p-1}{p'} \left( (\frac{C}{C-1})^2 \cdot \frac{(\frac{2C-1}{C})^{-\frac{1}{p-1}}+(\frac{1}{C})^{-\frac{1}{p-1}}}{2}-(\frac{C}{C-1})^2\right)
$$
\end{enumerate}
\end{prop}
\begin{proof}[The first part]
Suppose $f$ is increasing. Then
$$
 \int_{-1}^{1} (1-|t|)  f(x + \varepsilon t) dt \geqslant \int_{0}^1 (1-|t|)  f(x + \varepsilon t) dt \geqslant f(x)\int_{0}^1 (1-|t|)dt.
$$
If $f$ is decreasing then we consider the integral over $(-1,0)$, which finishes the proof of the first part.
\end{proof}
\begin{proof}[The second part]
Let $x(s)=\frac{(x-t)(1-s)+(x+t)(1+s)}{2}$, and $a(s)=A(x(s))$. Then we would like to estimate the quantity
\begin{multline}
a(0)-\frac{a(1)+a(-1)}{2} = -\frac{1}{2} \int_{-1}^1 (1-|s|)a''(s)ds = \\=-\frac12 \cdot (2t)^2 \int_{-1}^1 (1-|s|)A''(x(s))ds = -c\cdot t^2 \int_{-1}^1 (1-|s|)A''(x + st)ds.
\end{multline}

Thus,
$$
A(x)-\frac{A(x-t)+A(x+t)}{2} \leqslant -c\cdot t^2 A''(x),
$$
which is exactly what we want.

\end{proof}
\begin{proof}[The third part]
Due to the homogeneity, this inequality is equivalent to the following:
$$
f(u):=u^p - \frac{(u+1)^{p}+(u-1)^p}{2} + \beta u^{p-2} \geqslant 0, \; \; \; \; u>1
$$
We notice that $f$ is continuous, and $\lim\limits_{u\to\infty} \frac{f(u)}{u^{p-2}}$ is finite. Therefore, such $\beta$ exists.
\end{proof}
\begin{proof}[The fourth part]
Again using homogeneity, we reduce our problem to the following: the function
$$
f_0(u)=u^{-\frac{1}{p-1}} - \frac{(u-1)^{-\frac{1}{p-1}} - (u+1)^{-\frac{1}{p-1}}}{2}+\gamma u^{-2-\frac{1}{p-1}}
$$
should be non-negative, when $u\geqslant \frac{C}{C-1}$. Here $\gamma=\alpha \frac{p'}{p-1}$.

We multiply by $u^{2+\frac{1}{p-1}}$ and, denoting $v=u^{-1}$, we need
$$
f_1(v)=\frac{1}{v^2}-\frac{(1-v)^{-\frac{1}{p-1}}+(1+v)^{-\frac{1}{p-1}}}{2v^2} + \gamma \geqslant 0,
$$
or the function
$$
f(v)=\frac{1}{v^2}-\frac{(1-v)^{-\frac{1}{p-1}}+(1+v)^{-\frac{1}{p-1}}}{2v^2}
$$
should be bounded from below, whenever $0<v<\frac{C-1}{C}$.

We prove the following:
\begin{lemma}[Sublemma]
 $f(v)$ is decreasing.
\end{lemma}
If we prove the sublemma, we get
$$
f(v)\geqslant f\left(\frac{C-1}{C}\right),
$$
and, therefore,
$$
\gamma = -f\left(\frac{C-1}{C}\right).
$$

\end{proof}
\begin{proof}[Proof of sublemma]
 We prove this proposition by straightforward differentiation. First,
$$
v^2 f(v)=1-\frac{(1-v)^{-\frac{1}{p-1}}+(1+v)^{-\frac{1}{p-1}}}{2},
$$
and so
$$
2vf(v) + v^2 f'(v) = \frac{1}{p-1}\frac{(1+v)^{-1-\frac{1}{p-1}}-(1-v)^{-1-\frac{1}{p-1}}}{2},
$$
thus
$$
v^2 f'(v) = \frac{1}{p-1}\frac{(1+v)^{-1-\frac{1}{p-1}}-(1-v)^{-1-\frac{1}{p-1}}}{2} - \frac{2}{v}+\frac{(1-v)^{-\frac{1}{p-1}}+(1+v)^{-\frac{1}{p-1}}}{v}
$$

We would like to prove that $f'(v)<0$ or, equivalently, the right-hand side is negative. We multiply by $v$ to get (after simple algebra)
$$
v^3 f'(v) = (1+v)^{1-p'} \frac{p'+1}{2} + (1-v)^{1-p'}\frac{p'+1}{2} - ((1+v)^{-p'} + (1-v)^{-p'})\frac{1}{2(p-1}-2 =:\psi(v).
$$
Clearly, $\psi(0)=0$. Next,
$$
\psi'(v)=\frac{(1-p')(p'+1)}{2}(1+v)^{-p'} - \frac{(1-p')(p'+1)}{2} (1-v)^{-p'} + \frac{p'}{2(p-1)}((1+v)^{-1-p'}-(1-v)^{-1-p'}),
$$
$$
\psi''(v)=\frac{p'(p'+1)}{2(p-1)}\cdot v\cdot ((1+v)^{-2-p'}-(1-v)^{-2-p'}).
$$
Thus, $\psi''(v)\leqslant 0$, thus $\psi'(v)\leqslant \psi'(0)=0$, and so $\psi(v)\leqslant \psi(0)=0$, which is what we want.
\end{proof}

\subsection{Bellman Function Proof of Lemma \ref{main_Lemma}}
\label{s: proof main lemma}
We remind that $A(x)$ be a convex twice differentiable function on $(0,\infty)$ such that for every $x \in (0,\infty)$ second derivative of $A$ satisfies the following inequality for every $\varepsilon>0$
 \begin{equation}
 \label{condition_A''}
 \int_{-1}^{1} (1-|t|) A^{\prime\prime} (x + \varepsilon t) dt \geqslant Q A^{\prime\prime}(x)
 \end{equation}
 holds with some positive constant $Q$ uniformly on $x$ and $\varepsilon$.

 Then for every weight $w$ and an interval $J$
 \begin{equation}\label{Asum}
 \frac{1}{|J|} \sum_{I\in \mathcal{D}(J)} \left(\av{w}{I^+} - \av{w}{I^-}\right)^2 A^{\prime\prime}(\av{w}{I}) |I| \leqslant 8 \frac{1}{Q}  \left(\av{A(w)}{J} - A(\av{w}{J})\right).
 \end{equation}




\begin{proof}
Take a function of two variables $B(u,v)=v-A(u)$. Then, as we have proved,
$$
B(u,v)-\frac{B(u-t, v-s) + B(u+t, v+s)}{2} = -\left(A(u)-\frac{A(u-t)+A(u+t)}{2}\right) \geqslant \alpha t^2 A''(u),
$$
whenever $B$ is defined at points we write.

We now take a weight $w$. Then $\av{w}{I_+} + \av{w}{I_-}=2\av{w}{I}$, and so $\av{w}{I_\pm}=\av{w}{I}\pm t$. Therefore,
$$
B(\av{w}{J}, \av{A(w)}{J}) \geqslant \frac{B(\av{w}{J_+}, \av{A(w)}{J_+}) + B(\av{w}{J_-}, \av{A(w)}{J_-})}{2} + \alpha (\Delta_J w)^2 A''(\av{w}{J}).
$$
We rewrite this inequality in the following form:
$$
|J|B(\av{w}{J}, \av{A(w)}{J}) \geqslant |J_+|B(\av{w}{J_+}, \av{A(w)}{J_+}) + |J_-|B(\av{w}{J_-}, \av{A(w)}{J_-}) + \alpha (\Delta_J w)^2 A''(\av{w}{J})|J|.
$$
Now we repeat this estimate down to $n$-th descendants of $J$. We denote this family by $\mathcal{D}_n(J)$. We get
$$
|J|B(\av{w}{J}, \av{A(w)}{J}) \geqslant \sum_{I\in\mathcal{D}_n(J)} |I|B(\av{w}{I}, \av{A(w)}{I})+\alpha\sli_{k\leqslant n}\sli_{I\in \mathcal{D}_k(J)}(\Delta_I w)^2 A''(\av{w}{I})|I|.
$$
Using that $B\geqslant 0$ whenever $v\geqslant A(u)$, which in our case is just Jensen's inequality, we get
$$
|J|B(\av{w}{J}, \av{A(w)}{J}) \geqslant \alpha\sli_{k\leqslant n}\sli_{I\in \mathcal{D}_k(J)}(\Delta_I w)^2 A''(\av{w}{I})|I|.
$$
Since the last estimate is true for any $n$, we pass to the limit and get
$$
|J|B(\av{w}{J}, \av{A(w)}{J}) \geqslant \alpha\sli_{I\in \mathcal{D}(J)}(\Delta_I w)^2 A''(\av{w}{I})|I|.
$$
But this is exactly what we want. Our proof is finished.
\end{proof}

\subsection{Proof of the ``inverse'' lemma \ref{inv_lemma}} \label{s: proof inv lemma}
\begin{proof}
We follow the skim of the previous proof. Take a function
$$
B(u,v)=\alpha v + A(u).
$$
Then $B$ satisfies the following inequality:
$$
B(x)-\frac{B(x+t, y+s-t^2 A''(x)) + B(x-s, y-s-t^2 A''(x))}{2} = A(x)-\frac{A(x-t)+A(x+t)}{2} + \alpha t^2A''(x) \geqslant 0.
$$
The last inequality is true for $A(x)=x^p$ or $A(x)=x\log x$ without additional assumptions, or for $A(x)=x^{-\frac{1}{p-1}}$ or $A(x)=-\log(x)$, if $t<\frac{C-1}{C}x$.

We now take a weight $w$. If $w$ is doubling (which we need only for the second part), then there exists a constant $D(w)$, such that for any dyadic interval $J$ the following is true:
$$
\av{w}{J}\leqslant D(w) \av{w}{J_{\pm}}.
$$
If now $\av{w}{J_{\pm}}=\av{w}{J}\pm t=x\pm t$, then
$$
x\leqslant D(w) (x-t),
$$
which implies
$$
t\leqslant \frac{C-1}{C}x
$$
for $C=D(w)$.

We now denote $u_I=\av{w}{I}$ and $v_I = \frac{1}{|I|}\sli_{R\in \mathcal{D}(I)} (\Delta_R(w))^2 A''(\av{w}{R})|R|$.

We notice that if $u_{I_\pm}=u_I \pm t$ then
$$
v_I - \frac{v_{I_+}+v_{I_-}}{2} = (\Delta_I(w))^2 A''(\av{w}{I})=t^2 A''(\av{w}{I})=t^2A''(u_J).
$$
So, $v_{I_\pm}=v_I \pm s - t^2 A''(\av{w}{I})$. Therefore, by our inequality for $B$, we get
$$
B(u_J, v_J)-\frac{B(u_{J_+})+B(u_{J_-})}{2} \geqslant 0.
$$
By the usual procedure, we get
$$
|J|B(u_J, v_J)\geqslant \sli_{I\in \mathcal{D}_n(J)}|I|B(u_I, v_I).
$$
We now introduce sequence of step functions: for a fixed $n$ we take the family $\{I\colon I\in \mathcal{D}_n(J)\}$, and
$$
u_n(t) = u_I, \; t\in I,
$$
$$
v_n(t)=v_I, \; t\in I.
$$
Then the last inequality is the same as
$$
|J|B(u_J, v_J)\geqslant \ili_J B(u_n(t), v_n(t))dt.
$$
We now notice that $B(u,v)=\alpha v + A(u)\geqslant A(u)$, so
$$
|J|B(u_J, v_J)\geqslant \ili_J A(u_n(t))dt.
$$
By Fatou's lemma,
$$
|J|B(u_J, v_J)\geqslant \liminf_n \ili_J A(u_n(t))dt \geqslant \ili_J \liminf_n A(u_n(t))dt = \ili_J A(w(t))dt.
$$
The last is true since for almost every $t$, by the Lebesgue Differentiation Theorem, we have $u_n(t)\to w(t)$, and because $A$ is a continuous function.

Dividing by $|J|$, we finally get $\alpha v_J + A(u_J)\geqslant \av{A(w)}{J}$, which finishes our proof.
\end{proof}

\subsection{Proof of Main Theorem I}

\label{s:proofH1<eAinfty}
\subsubsection{Notation and definition of function $B$}
For a point $z=(x,y)\in \R^2$ we denote $[z]=xe^{-y}$. For any number $Q$, $Q>1$, we define the domain $\Omega_Q$ as follows:
$$
\Omega_Q = \{ z=(x,y)\colon 1\leqslant [z]\leqslant Q \},
$$
and the boundaries of $\Omega_Q$ are
\begin{align*}
&\Gamma=\{z\colon [z]=1\} \\
&\Gamma_Q = \{z \colon [z]=Q\}.
\end{align*}
With any point $z\in \Omega_Q$ we associate two numbers: $v$ and $a$. We take our point $z$ and consider the line $\ell(z)$, tangent to $\Gamma_Q$, that ``kisses'' $\Gamma_Q$ on the right-hand side from $z$. The point $\ell(z)\cap \Gamma_Q$ is denoted by $(a, \log\frac{a}{Q})$. Now we draw $\ell(z)$ to the left until it intersects $\Gamma$, and the point of intersection is denoted by $(v, \log(v))$. Notice that $v\leqslant x \leqslant a$.

More carefully, let $\gamma = \gamma(Q)$, $\gamma\leqslant 1$, be the smaller solution of equation
$$
\gamma-\log(\gamma)-1 = \log(Q).
$$
Then the line $\ell(z)$ is given by a formula
$$
y=\frac{\gamma\cdot x}{v} + \log(v)-\gamma.
$$
This equation defines a unique $v$, such that $v\leqslant x$. Moreover, $a$ is given by $v=\gamma \cdot a$.
We are ready to introduce the Bellman Function. We give an explicit formula:
$$
B_Q(z)=B_Q(x,y)=x\cdot\log(v)+\frac{x-v}{\gamma}.
$$
\begin{zamech}
The equation on $t$, $t-\log(t)=\log(u)$, it rather famous and developed. In the mathematical program Maple this solution can be obtained using a command
$$
-LambertW(-\frac{1}{u}).
$$
Several inequalities in next sections can be checked by graphing related functions. The second author wants to emphasize his gratitude to developers of Maple.
\end{zamech}
\subsubsection{Main theorems and discussion}
The following theorem was proved in ~\cite{BR}.
\begin{theorem}
The function $B_Q(z)$ has following properties:
\begin{enumerate}
\item $B_Q(v, \log(v)) = v\log(v)$.
\item $B$ is smooth in $\Omega_Q$, and locally concave in $\Omega_Q$. Namely, if $z_1, z_2\in \Omega_Q$, $z=sz_1 + (1-s)z_2$ for some $s\in [0,1]$ and $\{tz_1 + (1-t)z_2\}\subset \Omega_Q$ then $$
    B(z)\geqslant sB(z_1)+(1-s)B(z_2).
    $$
\item For every point $z=(x,y)\in \Omega_Q$ there exists a function $w$, $[w]_\infty \leqslant Q$, such that $\ave{w}=x$, $\ave{\log(w)} = y$, and $\ave{w\log(w)} = B(x,y)$.
\end{enumerate}
\end{theorem}

This theorem implies the following (see ~\cite{BR}).
\begin{theorem}
The following equality holds:
$$
B_Q(x,y)=\sup \{\ave{w\log(w)} \colon \ave{w}=x, \ave{\log(w)}=y, [w]_\infty\leqslant Q\}.
$$
\end{theorem}
We sketch the proof of this theorem.
\begin{proof}
The third property of $B$ implies that $B_Q(x,y)$ is not strictly bigger than the right-hand side. For the other direction, we take a point $z=(x,y)\in \Omega_Q$ and a function $w$, such that $\ave{w}=x, \ave{\log(w)}=y, [w]_\infty\leqslant Q$.
We now take two intervals $I_\pm$, such that $I_+ \cup I_- = I$, $I_+\cap I_- =$right end of $I_-$. We take
\begin{equation}\label{tralyalya}
z_\pm=(x_\pm, y_\pm)= (\av{w}{I_\pm}, \av{\log(w)}{I_\pm})\in \Omega_Q.
\end{equation}
Assuming that the interval $[z_-, z_+]$ lies in $\Omega_Q$, we write
$$
B(z)\geqslant \frac{|I_-|}{|I|}B(z_-) + \frac{|I_+|}{|I|}B(z_+).
$$
Repeating this procedure, we get
$$
B(z)\geqslant \sli_{n=1}^{N} \frac{|I^n|}{|I|}B(z_n),
$$
where $z_n = (\av{w}{I^n}, \av{\log(w)}{I^n})$. We now introduce a pair of step functions.
Let
\begin{align*}
&u_N(t) = \sli_{n=1}^N \av{w}{I^n} \chi_{I^n} (t),\\
&v_N(t)=\sli_{n=1}^N \av{\log(w)}{I^n} \chi_{I^n} (t).
\end{align*}
Then we have
$$
B(z)\geqslant \ili_I B(u_N(t), v_N(t)) dt.
$$
If $w$ is separated from $0$ and $\infty$ then, by the Lebesgue theorem, we get that
\begin{align*}
&u_N(t)\to w(t)\; \mbox{a.e.} \\
&v_N(t)\to \log(w(t)) \; \mbox{a.e.}.
\end{align*}
Therefore,
$$
B(z)\geqslant \ili_I B(w(t), \log(w(t)))dt = \ili_I w(t)\log(w(t))dt = \ave{w\log(w)}.
$$
In the chain above we used that $B$ is bounded on compact sets, so we can apply the Lebesgue Dominated Convergence Theorem, and the second property of the function $B$.
The proof is finished.
\end{proof}
\begin{zamech}
Careful reader can see two gaps in the proof above. First, we never introduced a proper procedure of choosing intervals $I_\pm$. And second, we focused on bounded functions $w$ (and separated from $0$) without saying anything about the general case. We refer to the paper ~\cite{BR}, where all details are given.
\end{zamech}
\begin{zamech}
We now point out the main difficulty of the dyadic case. In the proof above we had a formula \eqref{tralyalya}. We claimed that $z_\pm=(x_\pm, y_\pm)= (\av{w}{I_\pm}, \av{\log(w)}{I_\pm})\in \Omega_Q$. In the dyadic case though this can be claimed only if $I_\pm$ are {\bf dyadic} intervals! Therefore, we do not have any procedure of choosing $I_\pm$ except for splitting $I$ in two halves, et cetera. The main problem now is that we can never be sure that the segment $[z_-, z_+]$ lies entirely in the domain $\Omega_Q$.
\end{zamech}
After these two remarks we state the main theorem, that works for dyadic setting.
Let $B_Q$ be a function, described above, defined in the domain $\Omega_Q$. For any $Q_0>Q$ we define $\Omega_{Q_0}=\Omega_0$, $\gamma_{Q_0}=\gamma_0$, $v_{Q_0}=v_0$, $a_{Q_0}=a_0$, and $B_{Q_0}(z)=B_0(z)$ as we did for $Q$. Then
\begin{theorem}\label{largerB}
There exists a constant $C$, which does not depend on $Q$, and a number $Q_0$, such that $1<Q<Q_0<CQ$, and such that the function $B_0$ has the following additional property: whenever $z, z_+, z_- \in \Omega_Q$, and $z=\frac{z_+ + z_-}{2}$, the following inequality holds:
$$
2B_0(z)\geqslant B_0(z_+) + B_0(z_-).
$$

If $r=\sqrt{1-\frac{1}{Q}}$ then $Q_0$ is given by equation
$$
(1-r)\log(\gamma_0) + \frac{1-r}{\gamma_0} -(1-r) - (1-r)\log(1-r)-(1+r)\log(1+r)=0.
$$
\end{theorem}
\begin{zamech}
We notice that this equation defines $\gamma_0$, which immediately defines $Q_0$.
\end{zamech}
\begin{zamech}
The following thing happened. We claim that we can take a larger domain and a function $B_{Q_0}$, which is bigger than $B_Q$, and which has the property: if three points $z, z_\pm$, described above, lie in the {\bf small} domain $\Omega_Q$, then $2B_0(z)\geqslant B_0(z_+) + B_0(z_-)$, even though the interval $[z_-, z_+]$ does not lie even in $\Omega_{Q_0}$.
\end{zamech}
The fact that the solution $Q_0$ of the equation above can be bounded by $CQ$ will be proved later. To emphasize the difficulty of the problem we prove a lemma, that shows the difference of our problem from the problem solved in ~\cite{SlVa}.
\begin{lemma}
For any constant $C$, $C>0$, there exists a number $Q$, $Q>1$, and three points $z, z_\pm \in \Omega_Q$, such that $2z = z_+ + z_-$, and such that for some value of $t\in (0,1)$ we have the following:
$$
[tz_+ + (1-t)z_-] \geqslant CQ.
$$
\end{lemma}
\begin{zamech}\label{cannotenlarge}
This lemma shows that for any given $Q_0=CQ$ there may be three points in $\Omega_Q$, such that the interval $[z_-, z_+]$ does not lie entirely in $\Omega_{Q_0}$.

We note that in ~\cite{SlVa} such constant $C$ existed, which did not simplify the search for the very best constant, but would give us linear dependence on $[w]_\infty$ at once. If this was the case, we would simply take $Q_0 = CQ$, and since $z_\pm, z\in \Omega_Q$ implied $[z_-, z_+]\in \Omega_{Q_0}$, where $B_{Q_0}$ is locally concave, we would get $2B_{Q_0}(z)\geqslant B_{Q_0}(z_+) + B_{Q_0}(z_-)$.

Since such $C$ does not exist, we are forced to continue our investigation.
\end{zamech}
\begin{proof}
This is an easy calculation. In fact, these points are $z_- = (1-r, \log\frac{1-r}{Q})$, $z=(1, \log\frac{1}{Q})$, and $z_+=(1+r, \log(1+r))$.
\end{proof}

As a main consequence of Theorem \ref{largerB} we get the following.
\begin{theorem}
For every point $z=(x,y)\in \Omega_Q$ the following inequality holds:
$$
B_{0}(z)\geqslant \sup \{\ave{w\log(w)} \colon \ave{w}=x, \ave{\log(w)}=y, [w]^{d}_{\infty}\leqslant Q\}.
$$
\end{theorem}
\begin{proof}
We take a point $z$, $1\leqslant[z]\leqslant Q$, and a function $w$, such that $\ave{w}=x$,
$\ave{\log(w)}=y$, and $[w]_\infty^d \leqslant Q$.
\paragraph{Case 1: $w$ is bounded away from $0$ and $\infty$}
We take $I^0=I$, $I^1_{1,2}$ --- left and right halves of $I$. Then $I^2_{1,2,3,4}$ are quarters of $I$, et cetera. For every $k,n$ we have $z^k_n(\av{w}{I^k_n}, \av{\log(w)}{I^k_n}) \in \Omega_Q$, and every $z^k_n$ is a center of interval, that corresponds to ``sons'' of $I^k_n$. Therefore, for a fixed $k$,
$$
B_0(z)\geqslant \sli_n \frac{|I^k_n|}{|I|} B_0(z^k_n) = \ili_I B_0(u_k(t), v_k(t))dt,
$$
where
\begin{align*}
&u_k(t) = \sli_n \av{w}{I^k_n} \chi_{I^k_n} (t),\\
&v_k(t)=\sli_n \av{\log(w)}{I^k_n} \chi_{I^k_n} (t).
\end{align*}
Since $w$ is separated from $0$ and $\infty$, we get
\begin{align*}
&u_k(t) \to w(t), \; \mbox{a.e.}\\
&v_k(t) \to \log(w(t)), \; \mbox{a.e.}.
\end{align*}
Since we have countably many intervals $\{I^k_n\}_{n,k}$, the set $Z=\{z^k_n\}$ is a compact, and thus the function $B_0$ is bounded on $Z$. Therefore, by Lebesgue Dominated Convergence Theorem,
$$
B_0(z)\geqslant \ave{w\log(w)}.
$$
\paragraph{Case 2: arbitrary $w$}
We sketch the proof here, as it is the same as in ~\cite{BR}. We take
$$
w_n(t) = \begin{cases} n, &w(t)\geqslant n \\
                       w(t), &1\leqslant w(t)\leqslant n \\
                       1, &w(t)\leqslant 1.
                       \end{cases}
$$
Then, as it follows from ~\cite{RVV}, $[w_n]_{\infty}^d \leqslant Q$. By the previous case we get
$$
B_0(z)\geqslant \ave{w_n \log (w_n)}=\ili_{\{t\colon w(t)\geqslant 1\}} w_n \log(w_n).
$$
On the set $\{t\colon w(t)\geqslant 1\}$ the sequence $w_n(t) \log(w_n)(t)$ increases to $w(t)\log(w(t))$, and passing to the limit, we get
$$
B_0(z)\geqslant \ili_{\{t\colon w(t)\geqslant 1\}} w(t) \log(w(t))dt \geqslant \ili_{I} w(t) \log(w(t))dt =\ave{w\log(w)}.
$$
The last inequality holds simply because on the set $w(t)<1$ we have $w(t)\log(w(t)) \leqslant 0$.

Our proof is now finished.
\end{proof}

The rest of this section is devoted to the proof of the Theorem \ref{largerB}. The uncurious reader can skip this proof since it does not involve any weight theory.

\subsubsection{Proof of the Theorem \ref{largerB}: reminder}
At first we would like to remind the reader some notation. We fix a number $Q$, $Q>1$. In what follows the number $Q_0$ is always bigger than $Q$.

For every point $z=(x,y)$, such that $xe^{-y}\in [1,Q]$ we denote $[z]=xe^{-y}$. Moreover, numbers $\gamma_0$, $v=v_0$ and $a=a_0$ are defined implicitly by
\begin{align}
&\gamma_0 - \log(\gamma_0)=1+\log(Q_0),\\
&y=\frac{\gamma_0\cdot x}{v} + \log(v)-\gamma_0,\\
&a=\frac{v}{\gamma_0}.
\end{align}
In what follows points $z_\pm$ are such that $2z=z_+ + z_-$, and $v_\pm, a_\pm$ are defined as above for these points.
Our ``larger'' function is defined as:
$$
B_0(x,y)=x\cdot\log(v)+\frac{x-v}{\gamma_0}.
$$
Furthermore,
\begin{align*}
&\Gamma_Q=\{z\colon [z]=Q\},\\
&\Gamma_{Q_0}=\{z\colon [z]=Q_0\},\\
&\Gamma=\Gamma_1=\{z \colon [z]=1\}.
\end{align*}
We sometimes refer to $\Gamma_Q$ as to a $Q$-boundary and to $\Gamma_{Q_0}$ as to a $Q_0$-boundary.
We start with the following easy lemma.
\begin{lemma}
Suppose $F(x,y,x_+, y_+, x_-, y_-)=2B_0(x,y)-B_0(x_-, y_-) - B_0(x_+, y_+)$. If $F(x, y, x_+, y_+, x_-, y_-)\geqslant 0$ then for every number $C$, $C>0$, the following holds: $F(Cx, y+\log(C), Cx_+, y_++\log(C), Cx_-, y_- + \log(C))\geqslant 0$.
\end{lemma}
\begin{proof}
This lemma follows immediately from the homogeneity of $B_0$, namely, $B_0(Cx, y+\log(C)) = Cx\log(C)+CB_0(x,y)$.
\end{proof}
This lemma allows us to choose $C=\frac{1}{x}$ and always think that $x=1$.

We first start with positions of $z, z_\pm$ that are supposed to be ``worst'' in some sense. Later we shall see that in fact the next section is not needed at all. However, for the sake of completeness we keep it.

\paragraph{Remark about notation}

Abusing notation, we always denote by $\Delta$ the following expression:
$$
\Delta=2B_0(z)-B_0(z_+)-B_0(z_-).
$$
However, in different sections the same letter $\Delta$ will depend (and be differentiated) on different variables. We will always specify on which variables it depends.

\subsubsection{Proof of the Theorem \ref{largerB}. First step}
We start our investigation from the case when $z_\pm, z$ are on the boundary of $\Omega_Q$.
Since $z, z_\pm$ are going to be fixed, $\Delta$ will depend only on $Q_0$.

Our first case is when two of them are on $\Gamma_Q$ and the third is on $\Gamma$. Second case is when two of them are on $\Gamma$ and the third is on $\Gamma_Q$. Moreover $z\in \Gamma_Q$ always.

\paragraph{$z_- \in \Gamma_Q$ and $z_+\in \Gamma$.}\label{bound1}
We have $z=(1, \log\frac{1}{Q})$.

We denote $z_+ = (1+r, \log(1+r))$ and $z_- = (1-r, \log\frac{1-r}{Q})$, $r\geqslant 0$.
Then, since $2y=y_+ + y_-$, we obtain, $2\log\frac{1}{Q}=\log \frac{1-r^2}{Q}$, so $r^2 = 1-\frac{1}{Q}$, and thus $r=\sqrt{1-\frac{1}{Q}}$.
Then we have:
$$
z_- = (1-r, \log\frac{1-r}{Q}), \; \; z=(1, \log\frac{1}{Q}), \; \; z_+ = (1+r, \log(1+r)).
$$
%
We prove the following theorem.
\begin{theorem}
Take $Q_0=Q$. Then we get $\gamma_0=\gamma$, $B_0=B$, and $v$, associated to $Q$.
Denote
$$
\Delta=\Delta(Q)=2B(z)-B(z_-)-B(z_+).
$$
Then $\Delta\geqslant 0$.
\end{theorem}
We notice that now $\Delta$ depends on $Q$, and $Q$ is a variable, that is bigger than $1$.

This theorem, together with next lemma, gives us what we want.
\begin{lemma}\label{enl1}
For fixed points $z, z_\pm \in \Omega_Q$, $\Delta(Q_0)=2B_0(z)-B_0(z_-)-B_0(z_+)$ is an increasing function with respect to $Q_0$ on the set $\{Q_0\colon Q_0\geqslant Q\}$.
\end{lemma}
The second lemma shows that if our initial $B$ was ``concave'' enough, then the ``enlarged'' $B_0$ is also ``concave'' enough.
\begin{proof}[Proof of the Lemma \ref{enl1}]
By definition,
$$
z_- = (1-r, \log\frac{1-r}{Q}), \; \; z=(1, \log\frac{1}{Q}), \; \; z_+ = (1+r, \log(1+r)).
$$
We have points $v, v_\pm\in \Gamma$, associated with $z, z_\pm$ and calculated in the enlarged domain. Namely,
\begin{align}
&\gamma_0 - \log(\gamma_0)=1+\log(Q_0)\\
&\log\frac{1}{Q}=\frac{\gamma_0}{v} + \log(v)-\gamma_0\\
&\log\frac{1-r}{Q}=\frac{\gamma_0 (1-r)}{v_-} + \log(v_-)-\gamma_0,\\
&v_+ = 1+r.
\end{align}
In particular we see that $v_- = (1-r)v$.
Since
$$
B_0(z) = x\log(v) + \frac{x-v}{\gamma_0},
$$
and since $2x-x_+-x_- = 0$, one gets
\begin{multline}
\Delta(Q_0) = 2\log v - (1-r)\log(v_-) - (1+r)\log(v_+) - \frac{1}{\gamma_0}(2v-v_--v_+) =\\= 2\log v - (1-r)\log(v)-(1-r)\log(1-r) - (1+r)\log(1+r) - \frac{1}{\gamma_0}(2v-(1-r)v - (1+r)) = \\= (1+r)(\log(v) - \frac{v-1}{\gamma_0}) - (1-r)\log(1-r) - (1+r)\log(1+r).
\end{multline}
Last two terms do not depend on $Q_0$ at all, so we consider only
$$
f(Q_0) = \log(v) - \frac{v-1}{\gamma_0}.
$$
We clearly have
$$
\gamma_0^{\prime} - \frac{\gamma_0^{\prime}}{\gamma_0} = \frac{1}{Q_0},
$$
so
$$
\gamma_0^{\prime} = \frac{\gamma_0}{(\gamma_0-1)Q_0}.
$$
Differentiating the equality
$$
\log\frac{1}{Q}=\frac{\gamma_0}{v} + \log(v)-\gamma_0
$$
with respect to $Q_0$, we get
$$
0 = \frac{\gamma_0^{\prime}}{v} - \frac{\gamma_0}{v^{2}}v^{\prime} + \frac{v^{\prime}}{v} - \gamma_0^{\prime},
$$
so
$$
0 = \frac{v^{\prime}}{v} \left(1-\frac{\gamma_0}{v}\right) - \gamma_0^{\prime}\left(1-\frac{1}{v}\right),
$$
$$
\frac{v^{\prime}}{v} \frac{v-\gamma_0}{v} = \frac{v-1}{v}\frac{\gamma_0}{(\gamma_0-1)Q_0},
$$
$$
\frac{v^{\prime}}{v} = \frac{1-v}{v-\gamma_0}\frac{1}{1-\gamma_0}\frac{\gamma_0}{Q_0}.
$$
Now let us differentiate $f(Q_0)$. We remind that
$$
f(Q_0) = \log(v) - \frac{v-1}{\gamma_0} = \log(v) + \frac{1-v}{\gamma_0},
$$
so
\begin{multline}
f^{\prime}(Q_0) = \frac{v^{\prime}}{v} + \frac{-v^{\prime}\gamma_0 - \gamma_0^{\prime}(1-v)}{\gamma_0^2} = \\
 = \frac{1-v}{v-\gamma_0}\frac{1}{1-\gamma_0}\frac{\gamma_0}{Q_0} -\frac{v^{\prime}}{\gamma_0} - \frac{\gamma_0 (1-v)}{(\gamma_0-1)Q_0}\frac{1}{\gamma_0^2} = \frac{1-v}{v-\gamma_0}\frac{1}{1-\gamma_0}\frac{\gamma_0}{Q_0} - \frac{1-v}{v-\gamma_0}\frac{1}{1-\gamma_0}\frac{v}{Q_0} + \frac{1-v}{1-\gamma_0}\frac{1}{Q_0 \gamma_0} = \\
  = \frac{1-v}{(1-\gamma_0)Q_0} \left(\frac{\gamma_0}{v-\gamma_0} - \frac{v}{v-\gamma_0} + \frac{1}{\gamma_0}\right) =
  \frac{1-v}{(1-\gamma_0)Q_0} \left(\frac{1}{\gamma_0}-1\right) \geqslant 0,
\end{multline}
since $v<1$, and $\gamma_0 < 1$.

This finishes the proof.
\end{proof}

\begin{proof}[Proof of the theorem]
We go back to $Q$, $\gamma$, $B$, and $v$, calculated for $\gamma$. We remind that in the statement of the theorem $Q_0=Q$.

Recall that
$$
z_- = (1-r, \log\frac{1-r}{Q}), \; \; z=(1, \log\frac{1}{Q}), \; \; z_+ = (1+r, \log(1+r)).
$$
Our $v, v_\pm$ can be written explicitly in terms of $\gamma$.
Indeed,
\begin{align}
&v_- = \gamma(1-r)\\
&v = \gamma \\
& v_+=1+r.
\end{align}
Then
\begin{multline}
\Delta=2B(z)-B(z_-)-B(z_+)=2(\log(\gamma) + \frac{1-\gamma}{\gamma}) - ((1-r)\log(\gamma(1-r)) + \frac{1-r-(1-r)\gamma}{\gamma}) - (1+r)\log(1+r) = \\
 = 2\log(\gamma) + \frac{2}{\gamma} - 2 - (1-r)\log(\gamma) - (1-r)\log(1-r) - \frac{1-r}{\gamma} + (1-r)-(1+r)\log(1+r) = \\
 = (1+r)\log(\gamma) + \frac{1+r}{\gamma}-(1+r) - (1-r)\log(1-r) - (1+r)\log(1+r).
\end{multline}
We notice that $1-r^2 = \frac{1}{Q}$, so $\log(1+r) = \log\frac{1}{Q} - \log(1-r)$. Therefore,
\begin{multline}
\Delta = (1+r)\left(\log(\gamma) + \frac{1}{\gamma}-1 - \log\frac{1}{Q}\right) - (1-r)\log(1-r)+(1+r)\log(1-r) = \\
=(1+r)\left(\log(\gamma)+\log(Q) + \frac{1}{\gamma}-1\right) + 2r\log(1-r) = (1+r)\left(\gamma + \frac{1}{\gamma} - 2\right) + 2r\log(1-r).
\end{multline}

We would like to say that $\Delta\geqslant 0$. Surprisingly, we can do it. Here is the chain of awful estimates.
We denote
$$
f(Q)=\Delta.
$$
Notice that
$$
\gamma^{\prime} (Q)=\frac{\gamma}{Q(\gamma-1)} \; \; \; r^{'}(Q)=\frac{1}{2rQ^2}.
$$
The last one is true since $r^2-1 = -\frac{1}{Q}$.

We notice that if $Q=1$ then $r=0$ and $\gamma=1$, so $f(1)=0$. We claim that $f^{\prime}(Q)\geqslant 0$, which will give the desired result.

We have
$$
f^{\prime}(Q)=(\gamma+\frac{1}{\gamma}-2)\frac{1}{2rQ^2} + (1+r)(1-\frac{1}{\gamma^2}) \frac{\gamma}{Q(\gamma-1)} + (2\log(1-r)-\frac{2r}{1-r})\frac{1}{2rQ^2}.
$$
We notice that $\gamma+\frac{1}{\gamma}-2\geqslant 0$ and we throw it away. Therefore,
$$
f^{\prime}(Q)\geqslant (1+r)(1-\frac{1}{\gamma^2}) \frac{\gamma}{Q(\gamma-1)} + (2\log(1-r)-\frac{2r}{1-r})\frac{1}{2rQ^2} = \frac{1+r}{Q}\frac{\gamma+1}{\gamma}+\frac{\log(1-r)}{rQ^2} - \frac{1}{Q^2}\frac{1}{1-r}.
$$
We now use that
$$
\frac{1}{1-r}=\frac{1+r}{1-r^2}=Q(1+r),
$$
thus
\begin{multline}
f^{\prime}(Q)\geqslant \frac{1}{Q}\left[ 1+r + \frac{1+r}{\gamma}+\frac{\log(1-r)}{rQ} - \frac{Q(1+r)}{Q}\right] =\\= \frac{1}{Q}\left[1+r + \frac{1}{Q\gamma(1-r)} + \frac{\log(1-r)}{rQ}-(1+r)\right] = \\=\frac{1}{Q^2r}\left[\frac{r}{\gamma(1-r)} + \log(1-r)\right].
\end{multline}
Finally, $0\leqslant 1-r <1$, so $\frac{r}{\gamma(1-r)}>\frac{r}{\gamma}$, and therefore
$$
f^{\prime}(Q)\geqslant \frac{1}{Qr^2}\left[ \frac{r}{\gamma}+\log(1-r)\right].
$$
We now denote
$$
g(Q)=\frac{r}{\gamma} + \log(1-r).
$$
Again, $g(1)=0$. We are going to prove that $g^{\prime}(Q)\geqslant 0$. Indeed,
$$
g^{\prime}(Q)=\frac{1}{\gamma}\frac{1}{2rQ^2}-\frac{r}{\gamma^2}\frac{\gamma}{Q(\gamma-1)} - \frac{1}{1-r}\frac{1}{2rQ^2} = \frac{1}{2rQ^2}\left[ \frac{1}{\gamma}-\frac{1}{1-r} - \frac{2r^2 Q}{\gamma(\gamma-1)}\right].
$$
But $r^2Q = (1-\frac{1}{Q})Q = Q-1$, which implies
\begin{multline}
g^{\prime}(Q) = \frac{1}{2rQ^2} \left[ \frac{1}{\gamma}-\frac{1}{1-r} - (Q-1)\frac{2}{\gamma(\gamma-1)}\right] = \\=\frac{1}{2rQ^2} \left[ \frac{1}{\gamma}-\frac{1}{1-r} - 2(Q-1)(\frac{1}{\gamma-1} - \frac{1}{\gamma})\right] =\\= \frac{1}{2rQ^2} \left[\frac{1}{\gamma} - \frac{1}{1-r} - 2\frac{Q-1}{\gamma-1} + \frac{2Q}{\gamma} - \frac{2}{\gamma}\right]
\end{multline}
Again $\frac{1}{1-r}=Q(1+r)$, so
$$
2rQ^2 \cdot g^{\prime}(Q) = -\frac{1}{\gamma} + 2\frac{Q-1}{1-\gamma}+\frac{2Q}{\gamma}-Q(1+r) = \frac{Q-1}{\gamma} + 2\frac{Q-1}{1-\gamma} + Q(\frac{1}{\gamma}-r-1).
$$
First two terms are clearly non negative. To check that the last one is non negative we do the following. $\gamma$ satisfies the equation $\vf(t)-\log(Q)=1$, where $\vf(t)=t-\log(t)$. $\vf$ is a decreasing function if $t\in (0,1]$. So if we prove that $\vf(\frac{1}{1+r})-\log(Q)\leqslant 1$ then we get that $\frac{1}{1+r}\geqslant \gamma$, which means that $\frac{1}{\gamma}\geqslant 1+r$. Thus,
$$
\vf(\frac{1}{1+r})-\log(Q) = \frac{1}{1+r}+\log(1+r)-\log(Q);
$$
the derivative of this expression is
$$
\left(\frac{1}{1+r}-\frac{1}{(1+r)^2}\right)\frac{1}{2rQ^2} - \frac{1}{Q} = \frac{1}{Q}\left[\frac{r}{(1+r)^2}\frac{1}{2rQ^2}-1\right] = \frac{1}{Q}\left[\frac{1}{2Q^2(1+r)^2} - 1\right].
$$
Since $Q>1$, $r>0$, we have $2Q^2(1+r)^2>2$, so the derivative is negative, and therefore
$$
\vf(\frac{1}{1+r})-\log(Q)\leqslant  
 \vf(1)=1.
$$
This completes the proof.
\end{proof}
\paragraph{$z_-\in \Gamma$, $z_+\in \Gamma_Q$.}\label{bound2}

In this case we still have $r=\sqrt{1-\frac{1}{Q}}$, but
$$
z_- = (1-r, \log(1-r)), \; \; z=(1, \log\frac{1}{Q}), \; \; z_+=(1+r, \log\frac{1+r}{Q}),
$$
and
\begin{align}
&v_- = 1-r, \\
&v = \gamma\\
&v_+ = \gamma(1+r).
\end{align}
So,
\begin{multline}
\Delta(Q) = 2\log(\gamma) + \frac{2}{\gamma}-2 -(1-r)\log(1-r) - ((1+r)\log(\gamma(1+r)) + \frac{1+r-(1+r)\gamma}{\gamma}) = \\=2\log(\gamma)+\frac{2}{\gamma}-2-(1-r)\log(1-r) - (1+r)\log(\gamma)-(1+r)\log(1+r) - \frac{1+r}{\gamma} + (1+r) =\\= (1-r)\log(\gamma) + \frac{1-r}{\gamma} -(1-r) - (1-r)\log(1-r)-(1+r)\log(1+r)
\end{multline}
Unfortunately, this expression is negative. To prove it one can take very large $Q$ and write the asymptotic of everything. We have no intention to do it. However, curious reader can draw the graph of $\Delta(Q)$ in, say, Maple, and see that the function is negative. We now fix our choice of $Q_0$.
\begin{defin}
We define $Q_0$ as a solution of $\Delta(Q_0)=0$, such that $Q_0\geqslant Q$.
\end{defin}
We notice that this choice of $Q_0$ is as written in the Theorem \ref{largerB}.

This definition leaves two questions: if such a $Q_0$ exists and, more complicated, if there is a uniform estimate $Q_0 \leqslant C \cdot Q$, where $C$ doesn't depend on $Q$. Fortunately, the answer is ``yes'' to both questions. We prove the following lemma.

\begin{lemma}
If $z, z_\pm$ as above, then for every point $(u,v)\in [z_-, z_+]$ the following holds: $ue^{-v}\leqslant CQ$, where $C$ is some uniform constant.
\end{lemma}
This lemma, indeed, shows that if we take $Q_0 = CQ$ then the function $B_{Q_0}$ will be locally concave in the domain $\Omega_{Q_0}$, and the line segment $[z_-, z_+]$ lies in this domain. Since $\Delta(Q)\leqslant 0$ and $\Delta(CQ)\geqslant 0$, we immediately get that between $Q$ and $CQ$ there is some $Q_0$ for which $\Delta(Q_0)=0$.
\begin{proof}[Proof of the Lemma]
The segment $[z_-, z_+]$ has a parametrization $u(t)=tx_+ (1-t)x_-$, $v(t)=ty_+ + (1-t)y_-$. Then
$$
\vf(t)=u(t)\exp(-v(t)) = (t(x_+-x_-) + x_-)\exp(-t(y_+ - y_-) - y_-).
$$
We would like to prove that $\vf(t)\leqslant CQ, \; t\in [0,1]$. We have
first of all, $\vf(0)=1$, $\vf(1)=Q$, so we need to check local extrema.
$$
\vf^{\prime}(t) = (x_+ - x_-)\exp(\ldots) - (y_+ - y_-) (t(x_+ - x_-)+x_-) \exp(\ldots).
$$
If
$$
\vf^{\prime}(t_*)=0
$$
then
$$
\frac{x_+ - x_-}{y_+ - y_-} =x_- + t_* (x_+ - x_-),
$$
so
$$
t_* = \frac{1}{y_+ - y_-} - \frac{x_-}{(x_+ - x_-)},
$$
or
$$
t_*(y_+ - y_-) = 1 - (y_+ - y_-)\frac{x_-}{x_+ - x_-}.
$$
Therefore,
$$
\vf(t_*) = \frac{x_+ - x_-}{y_+ - y_-} \exp\left((y_+ - y_-)\frac{x_-}{x_+ - x_-} - 1 - y_-\right).
$$
We now plug our $x_\pm$ and $y_\pm$. First of all,
$$
x_\pm = 1\pm r,
$$
so $x_+ - x_- = 2r$. Also $y_+ - y_- = \log\frac{1+r}{Q} - \log(1-r) = \log\frac{1+r}{Q(1-r)}$.
We notice that
$$
Q(1-r) = Q(1-\sqrt{1-\frac{1}{Q}}) = Q-\sqrt{Q^2 - Q} = \frac{Q}{Q+\sqrt{Q^2 - Q}} \asymp 1.
$$
So $y_+ - y_- \asymp 1$. As this proof doesn't involve any deep ideas, we finish is briefly. First of all, we are interested in large $Q$, because for bounded $Q$ we can always find a uniform $C$. So,
$$
r=\sqrt{1-\frac{1}{Q}} \sim 1-\frac{1}{2Q}\sim 1,
$$
and $x_- = 1-r \sim \frac{1}{2Q}$, $y_- \sim \log\frac{1}{2Q}$. So,
$$
\vf(t_*) \asymp 2(1-\frac{1}{2Q})\exp\left(C \cdot \frac{1}{2Q}\frac{1}{2} - 1 - \log\frac{1}{2Q}\right)\asymp 2 \exp(\log(2Q)) \asymp Q.
$$
This finishes our proof.
\end{proof}
\paragraph{$z_\pm \in \Gamma$}\label{bound3}

In this case we change our choice of $r$. We have $z_\pm = (1\pm r, \log(1\pm r))$, $z=(1, \log\frac{1}{Q})$. Since $\log(1-r^2) = 2\log\frac{1}{Q}$, we get $1-r^2 = \frac{1}{Q^2}$, or $r=\sqrt{1-\frac{1}{Q^2}}$.

As in the first case, we prove two propositions.
\begin{lemma}
$\Delta(Q)\geqslant 0$ and for every $Q_0\geqslant Q$ we have $\Delta(Q_0)\geqslant \Delta(Q)$.
\end{lemma}
\begin{proof}
We start from the second fact. We always have $v_\pm = 1\pm r$, and so
$$
\Delta(Q_0) = 2\log v + 2\frac{1-v}{\gamma_0} - (1-r)\log(1-r) - (1+r)\log(1+r).
$$
We have already seen that the sum of first two terms increase when $Q_0$ increase, and last two terms do not depend on $Q_0$.

For the first part, notice that when $Q_0=Q$ we have $v=\gamma$, and so
$$
\Delta(Q) = 2\log(\gamma) + 2\frac{1-\gamma}{\gamma} - (1-r)\log(1-r)-(1+r)\log(1+r).
$$
We have
$$
\gamma^{\prime} = \frac{\gamma}{Q(\gamma-1)}, \; \; \; r^{\prime} = \frac{1}{rQ^3}.
$$
The last one is new because $r$ is different from first two cases.
So,
$$
\Delta^{\prime} = \frac{2}{Q\gamma} + \frac{1}{rQ^3}\log\frac{1-r}{1+r} = \frac{2}{\gamma Q} \frac{1}{rQ^3} \log \frac{(1-r)^2}{1-r^2} = \frac{2}{Q}\left[ \frac{1}{\gamma} + \frac{\log(Q-\sqrt{Q^2-1})}{Q\sqrt{Q^2-1}}\right].
$$
We leave the proof that this expression is positive as an easy exercise. Then $\Delta(Q)\geqslant \Delta(1) = 0$, and we are done.
\end{proof}
\subsubsection{Proof of the Theorem \ref{largerB}: change of variables}
\paragraph{Discussion}

We remind the reader that in the general case we basically have four variables: $x_\pm$ and $y_\pm$. Then the center point $z=(1, y)$ is given by $2=x_+ + x_-$ and $2y=y_+ + y_-$. The first equation lets us get rid of $x_-$, and so we have three variables: $x_+$, $y_-$, and $y_+$. These variables have rather sophisticated domain. Here are the inequalities that define this domain:
\begin{align*}
& x_+ e^{-y_+}\in [1, Q] \\
& (2-x_+)e^{-y_-} \in [1,Q]\\
& e^{-\frac{y_+ + y_-}{2}} \in [1,Q].
\end{align*}
This domain is somewhat inconvenient for us. The ``explanation'' is the following. We want to minimize some function on this domain. In the interior we will be able to do it, but then we should switch to the boundary, that is pretty ``curved''.

It would be more convenient to introduce variables, for example, $x_+ e^{-y_+}$ and $x_- e^{-y_-}$. Their domain is $[1, Q]\times [1,Q]$, which looks better.

However, these variables are still not good enough. We are about to introduce the ``best'' variables.

\paragraph{New variables}

We denote
\begin{align*}
&\alpha = y-\log\frac{1}{Q_0} = y+\log(Q_0), \\
& \alpha_+ = y_+ - \log\frac{x_+}{Q_0},\\
&\alpha_- = y_- - \log\frac{2-x_+}{Q_0}.
\end{align*}
In fact, $\alpha, \alpha_\pm$ are vertical distance from the point $z, z_\pm$ to $\Gamma_{Q_0}$.
For a fixed $\alpha$ we have {\bf three} variables: $x_+, \alpha_+, \alpha_-$. They are related by equation
\begin{equation}\label{relalx}
2\alpha = \alpha_+ + \alpha_- + \log(x_+) + \log(2-x_+).
\end{equation}
So $\alpha_\pm, x_+$ are on some manifold, and to minimize a function of these three variables we should use Lagrange multipliers.
\paragraph{New domain}
Fix $\alpha \in [\log\frac{Q_0}{Q}, \log(Q_0)]$. We have following inequalities for $\alpha_{\pm}$ and $x_+$:
\begin{align*}
& \alpha_\pm \in [\log\frac{Q_0}{Q}, \log(Q_0)], \\
& x_+ \in [1,2),\\
& \alpha_+ + \alpha_- \geqslant 2\alpha.
\end{align*}
The last inequality follows from the fact that $\log(x_+) + \log(2-x_+)\leqslant 0$.

We also notice that in fact $x_+$ can not access all values from $[1,2)$. We do not pay attention to this fact, because from the \eqref{relalx}, $x_+$ can be calculated in terms of $\alpha_\pm$, and this is how Lagrange multipliers work.

So for any fixed $\alpha$ we pay attention only to the domain for $\alpha_\pm$. We state an easy lemma to understand this domain.
We notice that since $\alpha \geqslant \log\frac{Q_0}{Q}$, we get that the line $\alpha_+ + \alpha_- = 2\alpha$ intersects the square $[\log\frac{Q_0}{Q}, \log(Q_0)]\times [\log\frac{Q_0}{Q}, \log(Q_0)]$ (on the $(\alpha_-, \alpha_+$ plane).

We notice that domain will look differently when $\alpha \geqslant \log(Q_0)-\frac{1}{2}\log(Q)$ and when $\alpha$ is smaller than this number. The reason is that the vertex $\alpha_- = \log(Q_0)$, $\alpha_+ = \log\frac{Q_0}{Q}$ may find itself under the line $\alpha_+ + \alpha_- = 2\alpha$.

Therefore, the domain for $\alpha_-, \alpha_+$ looks as follows.
\begin{center}
\includegraphics[width=0.5\linewidth]{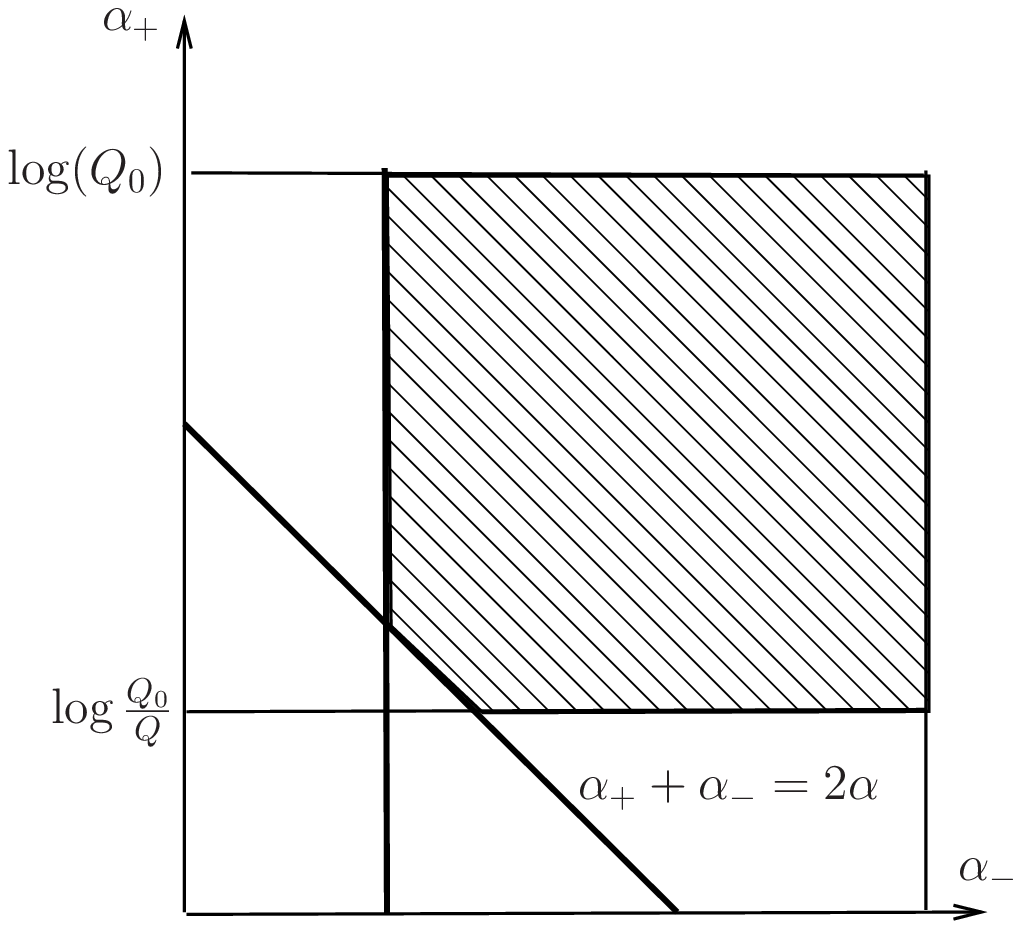}~
\includegraphics[width=0.5\linewidth]{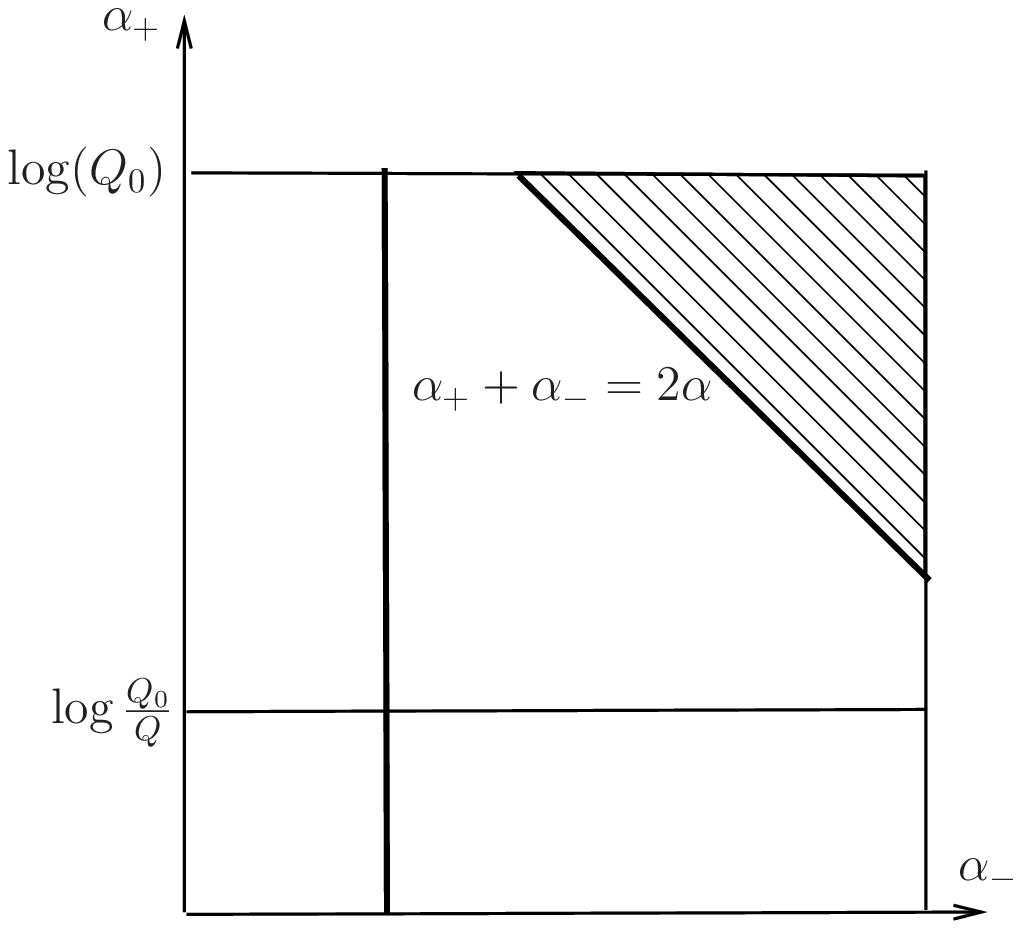}
\end{center}

We are going to study these two cases together. We shall prove that if the global minimum of $2B_0(z)-B_0(z_+)-B_0(z_-)$ is strictly negative then it is not obtained neither in the interior, nor in the interior of edges. Then we will investigate vertices. As the reader can see, edges and vertices, where $\alpha_+ + \alpha_- = 2\alpha$ correspond to vertical segments $[z_-, z_+]$ and therefore are trivial.

Thus, the second case will give us one interesting case: $\alpha_+ = \alpha_- = \log(Q_0)$, and the first case will give the same vertex and $\alpha_+ = \log\frac{Q_0}{Q}$, $\alpha_- = \log(Q_0)$.

After this short plan, let us give all details of searching for possible global minima.

\paragraph{Old variables and new variables}

We now need to recalculate old variables in terms of new ones. In particular, we need to relate $v$ and $v_\pm$ with $\alpha$ and $\alpha_\pm$ respectively. We will show in a moment that it is possible. The reason is that $\alpha$ is closely related to the number $a$, the first coordinate of a point, where the tangent line to $\Gamma_{Q_0}$, $\ell(z)$, ``kisses'' $\Gamma_{Q_0}$.

Let us proceed. Take any point $z=(x,y)$ in $\Omega_Q$. We for some time forget that $x=1$, and do calculations for arbitrary $x$. We do it because then they will work for $z_\pm$.

We say one more time that now $v$ and $a$ correspond to $Q_0$, so we should write $v_0$ and $a_0$, but to simplify the notation we do not do it.

We write the equation of the line $\ell(z)$, tangent to $\Gamma_{Q_0}$:
$$
y=\frac{\gamma_0 x}{v} + \log(v)-\gamma,
$$
so
$$
\alpha = y-\log\frac{x}{Q_0} = \frac{\gamma_0 x}{v}+\log(v)-\log(x)+\log(Q_0)-\gamma_0.
$$
Using the definition of $\gamma_0$, we obtain
$$
\alpha = \frac{\gamma_0 x}{v} + \log(v)-\log(x) - 1 - \log(\gamma_0) = \frac{\gamma_0 x}{v} - \log\frac{\gamma_0 x}{v} - 1.
$$
We now introduce a function
$$
f(t)=t-\log(t)-1, \; \; \; t>0.
$$
This function has already appeared in the definition of $\gamma_0$.
Function $f$ is decreasing from $+\inf$ to $0$ when $t\in (0, 1]$ and therefore has an inverse
$$
g(t)=f^{-1}(t), \; \; \; g\colon [0, \inf)\to (0, 1].
$$
\begin{center}
\includegraphics[width=0.5\linewidth]{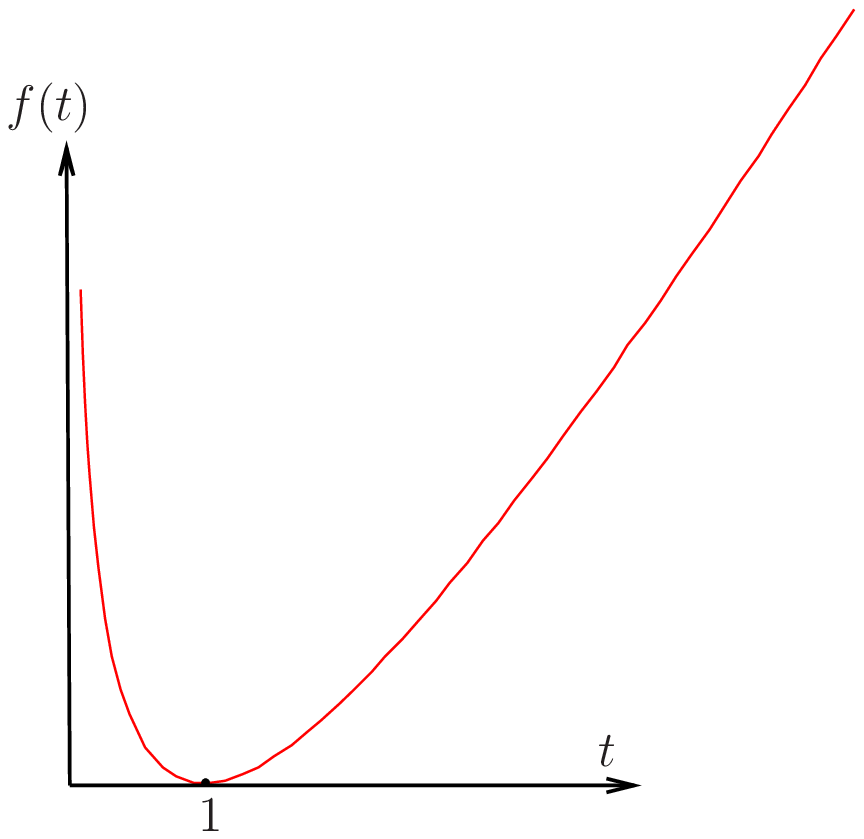}~
\includegraphics[width=0.5\linewidth]{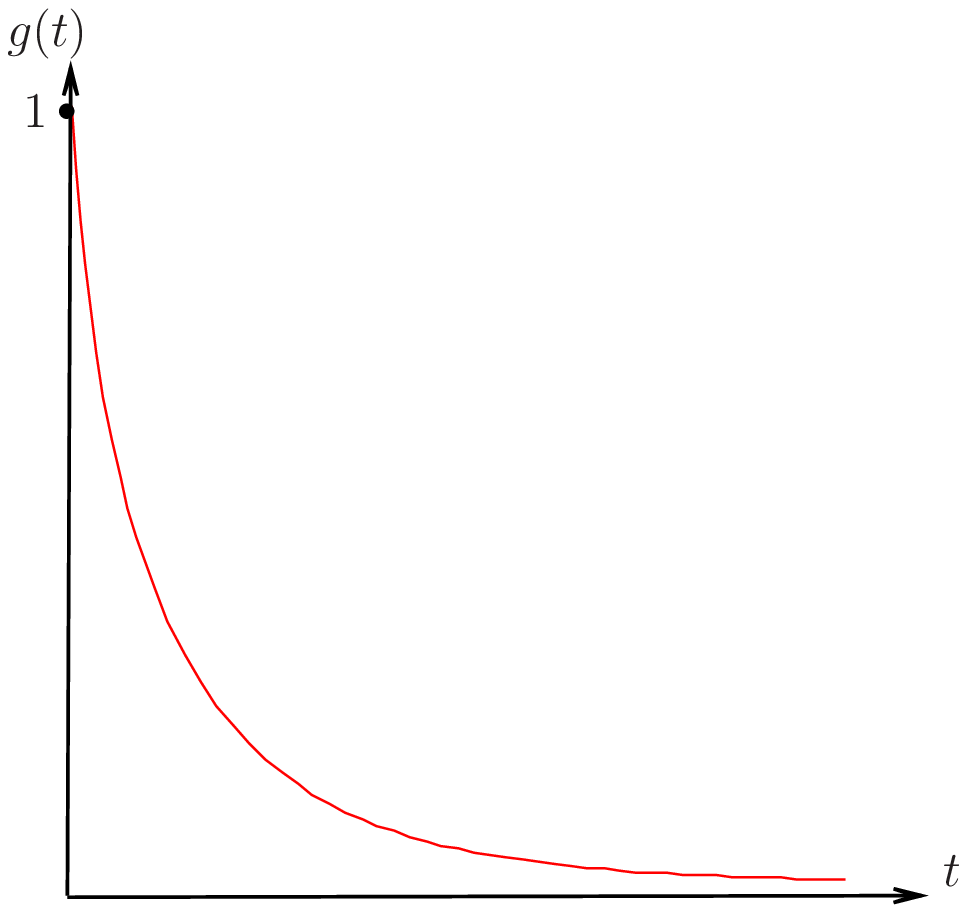}
\end{center}
We now have an equation
$$
\alpha = f(\frac{\gamma_0}{x}{v}).
$$
We notice that $x\leqslant a$, so $\gamma_0 x \leqslant \gamma_0 a = v$, and so $g(f(\frac{\gamma_0 x}{v})) = \frac{\gamma_0 x}{v}$. Therefore, we write
$$
\frac{\gamma_0 x}{v} = g(\alpha),
$$
or
$$
v=\frac{\gamma_0 x}{g(\alpha)}.
$$
In particular we notice that $g(\alpha) = \frac{x}{a}$. Basically this is the geometric meaning of $\alpha$.

The above equation with particular points $z=(1,y)$ and $z_\pm$ gives us
\begin{align*}
&v=\frac{\gamma_0}{g(\alpha)},\\
&v_+ = \frac{\gamma_0 x_+}{g(\alpha_+)}, \\
&v_- = \frac{\gamma_0 x_-}{g(\alpha_-)}.
\end{align*}

We are now ready to introduce the function that we want to minimize.

\paragraph{Function $\Delta$ in new variables}

We remind the reader that we fix $\alpha$ and have three variables $x_+$, $\alpha_+$, and $\alpha_-$ on the manifold
$$
2\alpha = \alpha_+ + \alpha_- + \log(x_+) + \log(2-x_+).
$$
We also remind the reader that $x_- = 2-x_+$ and $x=1$. Therefore, our function $\Delta$ will be
\begin{multline*}
\Delta(x_+, \alpha_+, \alpha_-) = 2B_0(z) - B_0(z_+) - B_0(z_-) = \\=2\left(\log(v) + \frac{1-v}{\gamma_0} \right) - \left(x_+ \log(v_+) + \frac{x_+ - v_+}{\gamma_0}\right) - \left((2-x_+)\log(v_-) + \frac{2-x_+ - v_+}{\gamma_0}\right).
\end{multline*}
We now want to rewrite the last expression in terms of $\alpha_\pm$ and $x_+$. We get
\begin{multline*}
\Delta = (2x\log v - x_+ \log v_+ - x_- \log v_-) - \frac{1}{\gamma_0}(2v-v_+ - v_-) = \\=2\log \frac{1}{g(\alpha)} - x_+ \log\frac{x_+}{g(\alpha_+)} - (2-x_+) \log\frac{2-x_+}{g(\alpha_-)} - \frac{2}{g(\alpha)} + \frac{x_+}{g(\alpha_+)} + \frac{2-x_+}{g(\alpha_-)}.
\end{multline*}
Due to the huge importance of this function, we write the final result separately:
$$
\Delta(x_+, \alpha_+, \alpha_-) = 2\log \frac{1}{g(\alpha)} - x_+ \log\frac{x_+}{g(\alpha_+)} - (2-x_+) \log\frac{2-x_+}{g(\alpha_-)} - \frac{2}{g(\alpha)} + \frac{x_+}{g(\alpha_+)} + \frac{2-x_+}{g(\alpha_-)}.
$$

We now start to minimize it. We prove the following theorem.
\begin{theorem}\label{minimum}
\begin{enumerate}
\item For a fixed $\alpha\leqslant \log(Q_0) - \frac{1}{2}\log(Q)$ the following holds:
$$
\min \Delta(x_+, \alpha_+, \alpha_-) = \min\left[0, \Delta(\widehat{x_+},\log(Q_0), \log(Q_0)), \Delta(\widetilde{x_+}, \log\frac{Q_0}{Q}, \log(Q_0))\right],
$$
where $\widehat{x_+}$ is a solution of equation $2\alpha = 2\log(Q_0) + \log(x_+) + \log(2-x_+)\; x_+\geqslant 1$, and $\widetilde{x_+}$ is a solution of equation $2\alpha = \log(Q_0) + \log\frac{Q_0}{Q} + \log(x_+)+\log(2-x_+), \; x_+\geqslant 1$.
\;\;\;\;\;\;\item For a fixed $\alpha> \log(Q_0) - \frac{1}{2}\log(Q)$ the following holds:
$$
\min \Delta(x_+, \alpha_+, \alpha_-) = \min\left[0, \Delta(\widehat{x_+},\log(Q_0), \log(Q_0))\right].
$$
\end{enumerate}
\end{theorem}
\begin{zamech}
We notice that the nonzero minimum may be attained only on vertices.
\end{zamech}
\paragraph{Derivatives of $\Delta$}
Before we form the lagrangian, let us find derivatives of $\Delta$ with respect to $\alpha_+$, $\alpha_-$ and $x_+$.
First of all,
$$
g^{\prime}(t) = \frac{1}{f^{\prime}(g(t))} = \frac{g(t)}{g(t)-1}.
$$
So,
$$
\frac{\partial \Delta}{\partial \alpha_+} = \frac{x_+}{g(\alpha_+)}\frac{g(\alpha_+)}{g(\alpha_+)-1} - \frac{x_+}{g(\alpha_+)^2} \frac{g(\alpha_+)}{g(\alpha_+)-1} = \frac{x_+}{g(\alpha_+)}.
$$
Similarly,
$$
\frac{\partial\Delta}{\partial \alpha_-} = \frac{2-x_+}{g(\alpha_-)}.
$$
Finally, we take the derivative with respect to $x_+$.
$$
\frac{\partial \Delta}{\partial x_+} = -\log\frac{x_+}{g(\alpha_+)} - 1 + \log\frac{2-x_+}{g(\alpha_-)} + 1 + \frac{1}{g(\alpha_+)} - \frac{1}{g(\alpha_-)} = -\log\frac{x_+}{g(\alpha_+)} + \log\frac{2-x_+}{g(\alpha_-)} + \frac{1}{g(\alpha_+)} - \frac{1}{g(\alpha_-)}.
$$

\paragraph{Step 1: interior of the domain}

Suppose we are in the interior of domain for $\alpha_+, \alpha_-$. We form a Lagrangian:
$$
L(x_+, \alpha_+, \alpha_-, \lambda) = \Delta(x_+, \alpha_+, \alpha_-)-\lambda\cdot(\alpha_+ + \alpha_- + \log(x_+) + \log(2-x_+) - 2\alpha).
$$
Differentiating it with respect to $\alpha_\pm$, we obtain
$$
\frac{x_+}{g(\alpha_+)}=\frac{2-x_+}{g(\alpha_-)} = \lambda.
$$
These equalities mean that
$$
g(\alpha_+) = \frac{x_+}{\lambda}, \; \; \; \; g(\alpha_-) = \frac{2-x_+}{\lambda}.
$$
Applying $f$ to both sides, and recalling that $f(g(t))=t$, we get
\begin{align*}
&\alpha_+ = f(\frac{x_+}{\lambda}) = \frac{x_+}{\lambda} - \log(x_+)+\log(\lambda)-1,\\
&\alpha_- = f(\frac{2-x_+}{\lambda})=\frac{2-x_+}{\lambda}-\log(2-x_+)+\log(\lambda)-1.
\end{align*}
Let us plug these equalities into
$$
\alpha_+ + \alpha_- - 2\alpha + \log(x_+) + \log(2-x_+) = 0.
$$
By direct calculation,
$$
\alpha = \frac{1}{\lambda} + \log(\lambda) - 1 = f(\frac{1}{\lambda}).
$$
We notice that $\frac{1}{\lambda} = \frac{g(\alpha_+)}{x_+} \leqslant g(\alpha_+) \leqslant 1$,
and so $g(f(\frac{1}{\lambda})) = \frac{1}{\lambda}$.

Notice that it would not be true if $\lambda$ was less than $1$.

So, $g(\alpha) = \frac{1}{\lambda}$ (in fact, from this equation we find $\lambda$).
Now we can calculate $\Delta$ at our point.

$$
\Delta = 2\log(\lambda)-x_+ \log(\lambda) - (2-x_+)\log(\lambda) - 2\lambda + \lambda +\lambda = 0.
$$
\paragraph{Conclusion}
From the calculation above we conclude the following: either the global minimum of $\Delta$ is zero, or the global minimum is obtained on the boundary.

\paragraph{Step 2: reduction to the case $\alpha_- \geqslant \alpha_+$.}
We now prove a technical but very useful lemma. It will show that it is sufficient to minimize $\Delta$ only on half of our domain, when $\alpha_-\geqslant \alpha_+$. This will show that we do not need to consider edges $\alpha_+=\log(Q_0)$ and $\alpha_- = \log\frac{Q_0}{Q}$, except for vertices.
\begin{lemma}
Fix $x_+$ and let
$$
\Delta(\alpha_+, \alpha_-)=\Delta(x_+,\alpha_+, \alpha_-) = 2\log \frac{1}{g(\alpha)} - x_+ \log\frac{x_+}{g(\alpha_+)} - (2-x_+) \log\frac{2-x_+}{g(\alpha_-)} - \frac{2}{g(\alpha)} + \frac{x_+}{g(\alpha_+)} + \frac{2-x_+}{g(\alpha_-)}.
$$
If $u>v$ then $\Delta(u,v) \geqslant \Delta(v,u)$.
\end{lemma}
\begin{zamech}[Discussion]
So, if $\alpha_+ > \alpha_-$ then $\Delta(\alpha_+, \alpha_-)\geqslant \Delta(\alpha_-, \alpha_+)$,
and so if the global minimum is attained on the boundary, it is for sure attained on the part when $\alpha_+ \leqslant \alpha_-$.
\end{zamech}
\begin{zamech}[Discussion]
Notice that this lemma is natural. As we have seen from the investigation of cases when $z, z_\pm$ are on the boundary, the worst case happens when $z_-\in \Gamma$ and $z_+\in \Gamma_Q$. This corresponds to $\alpha_- = \log(Q_0)$ and $\alpha_+ = \log\frac{Q_0}{Q}$, which is smaller than $\alpha_-$.
\end{zamech}
\begin{proof}
First, since $u>v$ we have $g(u)<g(v)\leqslant 1$. We denote $t=g(u)$ and $s=g(v)$, so $t<s\leqslant 1$. We have
\begin{multline*}
\Delta(u,v)-\Delta(v,u) = x_+ \log(t) +(2-x_+)\log(s)+\frac{x_+}{t} + \frac{2-x_+}{s} - \left( x_+ \log(s) + (2-x_+)\log(t) + \frac{x_+}{s} + \frac{2-x_+}{t} \right) =\\= (2x_+-2)\log(t)  +\frac{2x_+-2}{t} +(2-2x_+)\log(s) + \frac{2-2x_+}{s} = (2x_+-2) \left(\frac{1}{t}+\log(t) - \frac{1}{s}-\log(s)\right).
\end{multline*}
Denote $\vf(x)=\frac{1}{x} + \log(x)$. Then $\vf^{\prime}(x)=\frac{1}{x}-\frac{1}{x^2} = \frac{x-1}{x^2}<0$ when $x\leqslant 1$.
So, since $t<s\leqslant 1$, we get
$$
\Delta(u,v)-\Delta(v,u)\geqslant 0.
$$
\end{proof}
\paragraph{Step 3: edge $\alpha_+ + \alpha_- = 2\alpha $}
In this case $x_+=1$, and so $2-x_+=1$, and we have a vertical line segment $[z_-, z_+]$. It definitely lies entirely in $\Omega_Q$, where the function $B_0$ is locally concave. Therefore, $\Delta\geqslant 0$.

\paragraph{Step 4: edge $\alpha_- = \log(Q_0)$}

In this case our manifold is
$$
2\alpha = \alpha_+ + \log(Q_0) + \log(x_+) + \log(2-x_+).
$$
Keeping in mind that $\alpha_-$ is fixed and we can not differentiate with respect to it, we write the same Lagrangian as before, and take derivatives with respect to $\alpha_+$ and $x_+$. We have
$$
L(x_+, \alpha_+, \alpha_-, \lambda) = \Delta(x_+, \alpha_+, \alpha_-)-\lambda\cdot(\alpha_+ + \alpha_- + \log(x_+) + \log(2-x_+) - 2\alpha),
$$
and so
$$
\frac{x_+}{g(\alpha_+)}=\lambda.
$$
In particular, we again get that $\lambda \geqslant 1$.
We now differentiate with respect to $x_+$, and we get
$$
-\log\frac{x_+}{g(\alpha_+)}+\log\frac{2-x_+}{g(\alpha_-)} + \frac{1}{g(\alpha_+)}-\frac{1}{g(\alpha_-)} - \lambda\left(\frac{1}{x_+} - \frac{1}{2-x_+}\right) = 0.
$$
Using the equality $\frac{x_+}{g(\alpha_+)}=\lambda$, we get
$$
-\log(\lambda)+\log\frac{2-x_+}{g(\alpha_-)} - \frac{1}{g(\alpha_-)}+\frac{\lambda}{2-x_+}=0,
$$
and thus
$$
f(\frac{\lambda}{2-x_+}) = f(\frac{1}{g(\alpha_-)}).
$$
We notice that $2-x_+\leqslant 1$ and $\lambda \geqslant 1$, so $\frac{\lambda}{2-x_+}\geqslant 1$. Since $f(t)$ increases when $t\geqslant 1$, we get
$$
\frac{\lambda}{2-x_+} = \frac{1}{g(\alpha_-)}.
$$
The same equation we had when we were investigating the interior. This equation yields to $\Delta=0$.
\paragraph{Step 5: $\alpha_+ = \log\frac{Q_0}{Q}$}

This edge is more delicate. Here we differentiate with respect to $\alpha_-$ and $x_+$.

\begin{align*}
&\frac{2-x_+}{g(\alpha_-)}=\lambda,\\
&-\log\frac{x_+}{g(\alpha_+)} + \log\frac{2-x_+}{g(\alpha_-)} + \frac{1}{g(\alpha_+)} - \frac{1}{g(\alpha_-)} - \lambda(\frac{1}{x_+} - \frac{1}{2-x_+}) = 0.
\end{align*}
Substituting the first one into the second, we get
$$
-\log\frac{x_+}{g(\alpha_+)} +\log(\lambda) + \frac{1}{g(\alpha_+)} - \frac{\lambda}{x_+}=0.
$$
We make the following remark. Similarly to the previous step, we get $f(\frac{\lambda}{x_+}) = f(\frac{1}{g(\alpha_+)})$. But now we can not say that $\frac{\lambda}{x_+} \geqslant 1$, and so we can not conclude that $\frac{\lambda}{x_+} = \frac{1}{g(\alpha_+)}$. We show how to finish the proof without this conclusion. We also warn the reader that this proof would not work in the previous step because it is tied to the fact that $x_+$ is on the $Q$-boundary of $\Omega_Q$.

We now proceed as follows:
$$
\frac{x_+}{g(\alpha_+)} - x_+\log\frac{x_+}{g(\alpha_+)} = \lambda-x_+\log(\lambda).
$$
Substituting this in $\Delta$, we get
$$
\Delta = 2\log\frac{1}{g(\alpha)} + \lambda-x_+\log(\lambda) - 2\log(\lambda)-\frac{2}{g(\alpha)}+\lambda = 2(f(\lambda)-f(\frac{1}{g(\alpha)})=2(f(\frac{2-x_+}{g(\alpha_-)}) - f(\frac{1}{g(\alpha)})).
$$

Notice that if $\frac{2-x_+}{g(\alpha_-)} \geqslant \frac{1}{g(\alpha)} \geqslant 1$ then, due to the monotonicity of $f(t)$, we get that $\Delta\geqslant 0$. Therefore, we should prove that it is non negative when $\frac{2-x_+}{g(\alpha_-)}<\frac{1}{g(\alpha)}$. The following lemma proves this fact.
\begin{lemma}\label{nevilezet}
If $\frac{2-x_+}{g(\alpha_-)}<\frac{1}{g(\alpha)}$ then the line segment $[z_-, z_+]$ lies entirely in $\Omega_{Q_0}$ and, consequently, $2B_0(z) - B_0(z_+)-B_0(z_-)\geqslant 0$.
\end{lemma}
Before proving this lemma we need an observation, related to the geometry of $\Omega_Q$.

Take the point $z_+$, which in our case lies on $\Gamma_Q$, and take the tangent to $\Gamma_Q$. Since we assume that $x_+>x_-$ and $y_+>y_-$, we get the following: if the segment $[z_-, z_+]$ goes above this tangent line, then it lies entirely in $\Omega_Q$, and the fact stated in the lemma is true. So the only interesting case is when $[z_-, z_+]$ goes below this tangent. It means that it goes outside of $\Omega_Q$ nearby $z_+$, and then returns before it ``hits'' the point $z$. Therefore, the segment $[z_-, z]$ lies in $\Omega_Q$, so the only problem can occur between $z$ and $z_+$.
\begin{center}
\includegraphics[width=0.5\linewidth]{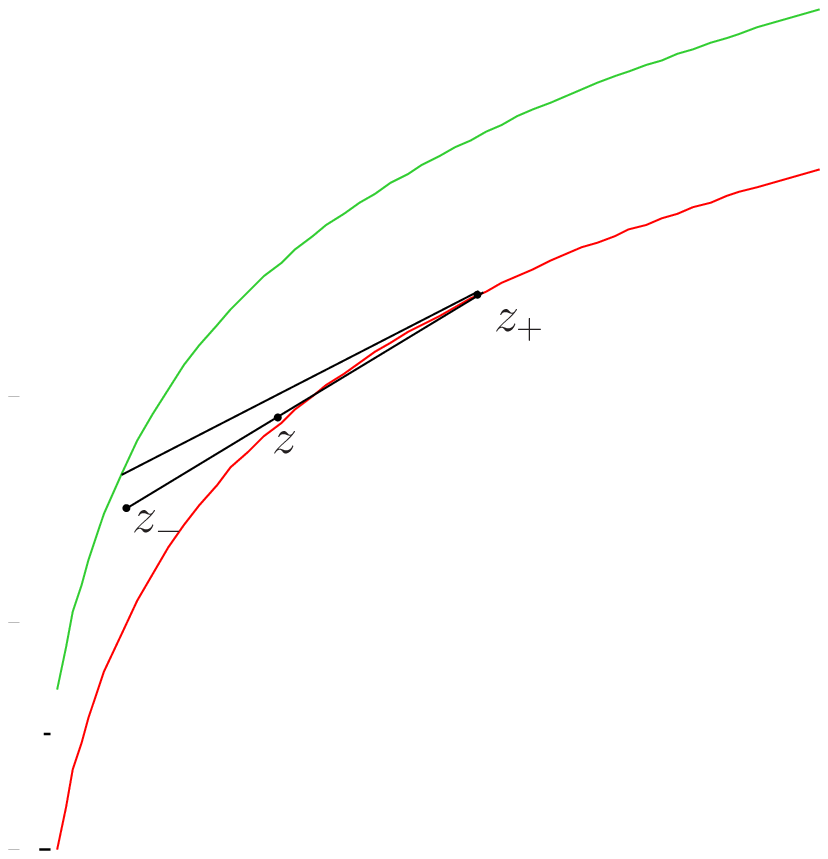}
\end{center}
Our lemma will be a consequence from the following one.
\begin{lemma}
Suppose $p\geqslant a \geqslant 1$, $\alpha = \frac{1}{a} - \log\frac{1}{a} - 1$ and $\alpha_+ = \frac{p}{a} -\log\frac{p}{a}-1$. If the line segment $[z_-, z_+]$ does not lie entirely in $\Omega_{Q_0}$ then $x_+ \geqslant p$.
\end{lemma}
\begin{proof}
Such $a$ and $p$ exist, because for every $u>0$ the equation $t-\log(t)-1=u$ has two solutions, one of which is less than $1$, and another is bigger than $1$.

We take our point $(1, y)$ and draw the tangent to $\Gamma_{Q_0}$ that goes to the right. Since the only possibility for $[z_-, z_+]$ to be outside of $\Gamma_Q$ is that part of $[z, z_+]$ is outside, we don't care about $z_-$ at all. If our $[z, z_+]$ goes above this tangent, then it's in $\Omega_{Q_0}$, and so the only ``bad'' case is when $[z, z_+]$ goes below. Suppose that the tangent ``kisses'' $\Gamma_{Q_0}$ at point $(a, \log\frac{a}{Q_0})$. Then the equation (in $(x_1, x_2)$ plane) is
$$
x_2 - y = \frac{1}{a} (x_1-1).
$$
Since $a$ satisfies the equation, and since $\alpha = y+\log(Q_0)$, we get
$$
\frac{1}{a} - \log\frac{1}{a}-1 = \alpha.
$$
Now take the point $(p, \log\frac{p}{Q})$ --- the point, where our tangent intersects $\Gamma_Q$ for the second time. This is the first time when our segment $[z, z_+]$ can return to $\Omega_Q$ (if it ever went out). Since $z_+$ is on the right-hand side from the ``return'' point, we have $x_+>p$. Let us find $p$.
\begin{center}
\includegraphics[width=0.5\linewidth]{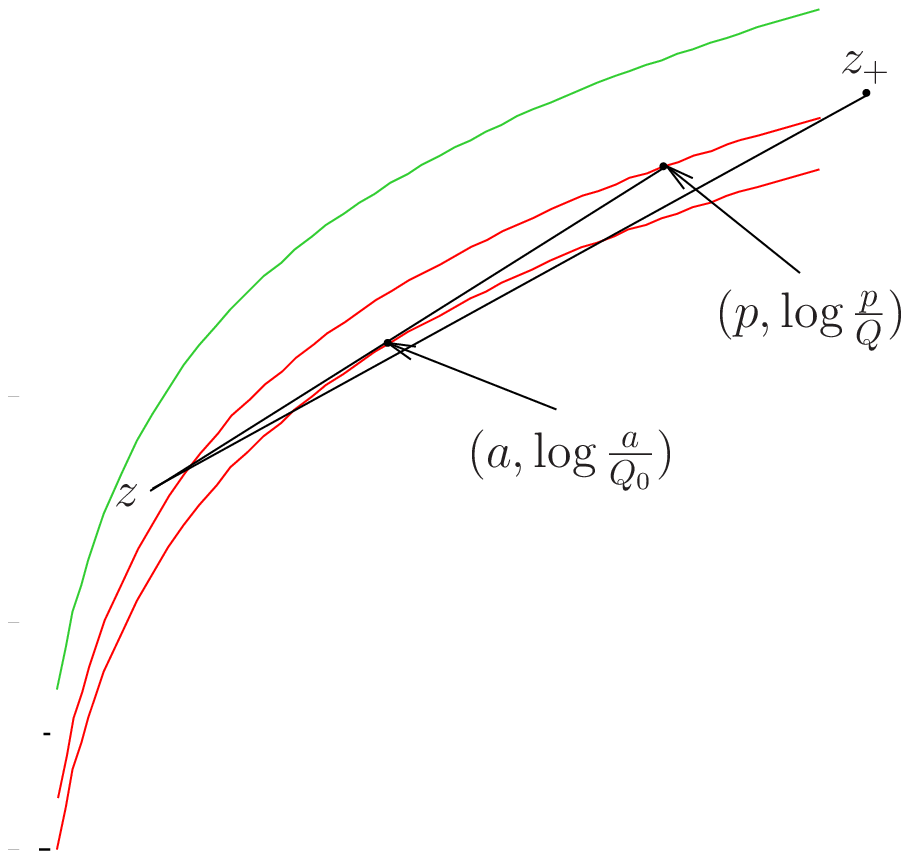}
\end{center}
We have
$$
\log\frac{p}{Q}-y = \frac{p}{a} - \frac{1}{a},
$$
and so, since $\alpha_+ = \log\frac{Q_0}{Q}$, we get
$$
\alpha_+ = \frac{p}{a} - \log\frac{p}{a} - 1.
$$
So both $a$ and $p$ are as in the statement, which finishes the proof.
\end{proof}
Now we prove the Lemma \ref{nevilezet}.
\begin{proof}
Suppose that $[z_-, z_+]$ does not lie in $\Omega_{Q_0}$. Then $x_+\geqslant p$, which implies $\frac{x_+}{a} \geqslant \frac{p}{a}>1$, so $f(\frac{x_+}{a})\geqslant \alpha_+$.
Next, we have
$$
2\alpha = \alpha_+ + \alpha_- +\log(x_+) + \log(2-x_+) \leqslant \frac{x_+}{a} - \log\frac{x_+}{a} - 1 + \alpha_- + \log(x_+) + \log(2-x_+).
$$
Recall that $\frac{1}{a}-\log\frac{1}{a}-1 = \alpha$, so
$$
2\alpha \leqslant \frac{x_+ - 1}{a} + \frac{1}{a} - \log\frac{1}{a} - 1 + \alpha_- + \log(2-x_+),
$$
thus
$$
\alpha \leqslant \frac{x_+ - 1}{a} + \alpha_- + \log(2-x_+).
$$
Using the equation for $a$ and $\alpha$ again, we get
$$
\frac{2-x_+}{a} - \log\frac{2-x_+}{a} - 1 \leqslant \alpha_-,
$$
so
$$
f(\frac{2-x_+}{a})\leqslant\alpha_-.
$$
We apply $g$ to both sides. We see that $2-x_+ = x_- \leqslant$, while $a>1$, so $g(f(\frac{2-x_+}{a}))=\frac{2-x_+}{a}$, therefore
$$
\frac{2-x_+}{a}\geqslant g(\alpha_-),
$$
and so
$$
\frac{2-x_+}{g(\alpha_-)} \geqslant a.
$$
But we know that $\alpha = f(\frac{1}{a})$ and $a>1$, so $g(\alpha)=\frac{1}{a}$, implies
$$
\frac{2-x_+}{g(\alpha_-)}\geqslant \frac{1}{g(\alpha)}.
$$
But this contradicts the assumption of our lemma.
\end{proof}

We now claim that the Theorem \ref{minimum} is proved. Indeed, the global minimum is either $0$, or attained on the boundary. On the boundary it is either again $0$, or attained on vertices. But vertices where $\alpha_+ + \alpha_- = 2\alpha$ give us non negative result, so the minimum may be attained only on vertices from the theorem.

\paragraph{Step 6: Vertex $\alpha_+ = \alpha_- = \log(Q_0)$}
In this case
$$
\alpha = \log(Q_0) + \frac{1}{2}\log(x_+ (2-x_+)).
$$
Let us get bounds for $x_+$. Clearly, $x_+\geqslant 1$, and this bound is accessible when $x_+=x_- = 1$. Since $\alpha \geqslant \log\frac{Q_0}{Q}$, we get $x_+(2-x_+)\geqslant \frac{1}{Q^2}$, which means that $x_+\leqslant 1+\sqrt{1-\frac{1}{Q^2}}$. As we know from the Section \ref{bound3}, this is also accessible when $x\in \Gamma_Q$. So, $x_+\in [1, 1+r]$, where $r=\sqrt{1-\frac{1}{Q^2}}$.

We now treat $\alpha_+$ as a function of $x_+$ and, therefore, our $\Delta$ becomes a function of $x_+$. We have
$$
\Delta(x_+)=2\log\frac{1}{g(\alpha)} -\frac{2}{g(\alpha)} - x_+\log(x_+) - (2-x_+)\log(2-x_+) -2\log\frac{1}{g(\alpha_+)} + \frac{2}{g(\alpha_+)}.
$$
From the same Section \ref{bound3} that $\Delta(1+r)\geqslant 0$. We intend to prove that $\Delta^{'}\leqslant 0$. Then we will be done with this case. We first notice that
$$
\frac{\partial \alpha}{\partial x_+} = \frac{1}{2}\left(\frac{1}{x_+}-\frac{1}{2-x_+}\right).
$$
Therefore,
\begin{multline*}
\Delta^{'}(x_+) = \frac{2}{g(\alpha)}\cdot \frac{1}{2} \left(\frac{1}{x_+} - \frac{1}{2-x_+}\right) - \log(x_+) + \log(2-x_+) = \frac{1}{g(\alpha)x_+} - \log(x_+) - \frac{1}{g(\alpha)(2-x_+)} + \log(2-x_+) = \\=\frac{1}{g(\alpha)x_+} + \log\frac{1}{g(\alpha)x_+} - \left(\frac{1}{g(\alpha)(2-x_+)} + \log\frac{1}{g(\alpha)(2-x_+)}\right).
\end{multline*}
The last equality is obtained by adding and subtracting $\log \frac{1}{g(\alpha)}$.
We notice that the function $s\mapsto \frac{1}{s}+\log\frac{1}{s}$ is decreasing, and $g(\alpha)x_+ \geqslant g(\alpha)(2-x_+)$. Therefore, $\Delta^{'}(x_+)\leqslant 0$, which finishes our proof in this case.
\paragraph{The vertex $\alpha_+ = \log\frac{Q_0}{Q}$, $\alpha_- = \log(Q_0)$.}
Now we set $\alpha_+ = \log\frac{Q_0}{Q}$ and $\alpha_- = \log(Q_0)$, so
$$
\alpha = \log(Q_0) - \frac{1}{2}\log(Q) + \frac{1}{2}\log(x_+(2-x_+)).
$$
Bounds for $x_+$ in this case are $1\leqslant x_+ \leqslant 1+r$, where $r=\sqrt{1-\frac{1}{Q}}$. We know from the Section \ref{bound2} that they are accessible, and that $\Delta(1+r)=0$ --- that is exactly our choice of $Q_0$, and this is the first and the only time when we use it. So again we would like to prove that $\Delta$ is decreasing. The difficulty is that now $\alpha_+\not= \alpha_-$, and so $\Delta$ does not have nice cancelations. We have
$$
\Delta(x_+) = 2\log\frac{1}{g(\alpha)} - \frac{2}{g(\alpha)} - x_+\log(x_+) - (2-x_+)\log(2-x_+)+x_+\log(g(\alpha_+)) + (2-x_+)\log(g(\alpha_-)) + \frac{x_+}{g(\alpha_+)} + \frac{2-x_+}{g(\alpha_-)},
$$
and so
$$
\Delta^{'}(x_+) = \frac{1}{g(\alpha)}\left(\frac{1}{x_+} - \frac{1}{2-x_+}\right) - \log(x_+)+\log(2-x_+) + \frac{1}{g(\alpha_+)} - \log\frac{1}{g(\alpha_+)} - \frac{1}{g(\alpha_-)} + \log\frac{1}{g(\alpha_-)}.
$$
From the investigation of previous vertex we know that
$$
\frac{1}{g(\alpha)}\left(\frac{1}{x_+} - \frac{1}{2-x_+}\right) - \log(x_+)+\log(2-x_+)\leqslant 0.
$$
This fact did not depend on the choice of $\alpha_\pm$. Finally, $\alpha_+ < \alpha_-$, so $g(\alpha_+)>g(\alpha_-)$, thus $\frac{1}{g(\alpha_-)} > \frac{1}{g(\alpha_+)}\geqslant 1$, and
$$
f(\frac{1}{g(\alpha_-)})>f(\frac{1}{g(\alpha_+)}).
$$
But this means exactly that
$$
\frac{1}{g(\alpha_+)} - \log\frac{1}{g(\alpha_+)} - \frac{1}{g(\alpha_-)} + \log\frac{1}{g(\alpha_-)}<0.
$$
Thus, our proof is finished.

\end{document}